\documentclass[11pt,hyp,]{amsart}
\usepackage{etex} 
\usepackage{graphicx,amssymb,upref}
\usepackage[letterpaper, margin=1.5in]{geometry} 

\let\<\langle
\let\>\rangle

\usepackage{amsmath}
\usepackage{amssymb}
\usepackage{amsthm}
\usepackage{amsfonts}
\usepackage{booktabs}
\usepackage[usenames,dvipsnames,svgnames]{xcolor}
\usepackage{mathrsfs}
\usepackage{mathtools}
\usepackage{pictexwd, dcpic}
\usepackage{stmaryrd}
\usepackage{mathabx}
\usepackage{graphicx}
\usepackage{pinlabel}
\usepackage[all]{xy}

\usepackage{tikz}
\usepackage{tikz-cd}
\usepackage{caption}
\usetikzlibrary{decorations.markings}

\newcommand{\CC}{\mathbb{C}}

\newcommand{\NN}{\mathbb{N}}

\newcommand{\PP}{\mathbb{P}}

\newcommand{\RR}{\mathbb{R}}

\newcommand{\ZZ}{\mathbb{Z}}

\newcommand{\vB}{\mathcal{B}}

\newcommand{\vT}{\mathcal{T}}

\newcommand{\GL}{\operatorname{GL}}
\newcommand{\Sp}{\operatorname{Sp}}

\newcommand{\LMod}{\operatorname{LMod}}
\newcommand{\Mod}{\operatorname{Mod}}
\newcommand{\PMod}{\operatorname{PMod}}
\newcommand{\SMod}{\operatorname{SMod}}

\newcommand{\lra}{\longrightarrow}
\newcommand{\sm}{\setminus}
\newcommand{\ol}[1]{\overline{#1}}
\newcommand{\wt}[1]{\widetilde{#1}}
\newcommand{\wh}[1]{\widehat{#1}}
\newcommand{\id}{\text{id}}

\newcommand{\abs}[1]{\left\lvert #1 \right\rvert}
\newcommand{\p}[1]{\smallskip\noindent {\bf #1.}} 

\theoremstyle{definition}

\newtheorem*{ques}{Question}

\theoremstyle{plain}
\newtheorem{thm}{Theorem}[section]
\newtheorem{lem}[thm]{Lemma}
\newtheorem{cor}[thm]{Corollary}

\newtheorem{sch}[thm]{Scholium}

\renewcommand{\emph}{\bf}

\title[Liftable Mapping Class Group of Balanced Superelliptic Covers]{The Liftable Mapping Class Group of Balanced Superelliptic Covers}

\author{Tyrone Ghaswala}
\address{Department of Pure Mathematics, University of Waterloo, Waterloo, ON, N2L 3G1,
Canada}
\email{ty.ghaswala@gmail.com}
\author{Rebecca R. Winarski}
\address{Department of Mathematical sciences, University of Wisconsin-Milwaukee, Milwaukee, WI 53211-3029, USA}
\email{rebecca.winarski@gmail.com}

\keywords{mapping class groups, branched covers, cyclic covers, spheres, abelianization, symmetric mapping class group}

\subjclass[2010]{Primary: 20F38; secondary: 20F34.}

\begin{document}

\begin{abstract}
The hyperelliptic mapping class group has been studied in various contexts within topology and algebraic geometry.  What makes this study tractable is that there is a surjective map from the hyperelliptic mapping class group to a mapping class group of a punctured sphere.  The more general family of superelliptic mapping class groups does not, in general, surject on to a mapping class group of a punctured sphere, but on to a finite index subgroup.  We call this finite index subgroup the liftable mapping class group.  In order to initiate the generalization of results on the hyperelliptic mapping class group to the broader family of superelliptic mapping class groups, we study an intermediate family called the balanced superelliptic mapping class group.  We compute the index of the liftable mapping class group in the full mapping class group of the sphere and show that the liftable mapping class group is independent of the degree of the cover.  We also build a presentation for the liftable mapping class group, compute its abelianization, and show that the balanced superelliptic mapping class group has finite abelianization.  Although our calculations focus on the subfamily of balanced superelliptic mapping class groups, our techniques can be extended to any superelliptic mapping class group, even those not within the balanced family.
\end{abstract}

\maketitle
\setcounter{tocdepth}{1}
\tableofcontents

\section{Introduction}

Let $\Sigma_g$ be a surface of genus $g$, and let $\zeta$ be a finite order homeomorphism of $\Sigma_g$ such that $\Sigma_g/\langle \zeta \rangle$ is homeomorphic to the sphere $\Sigma_0$.  The quotient map is a branched covering map $p:\Sigma_g \to \Sigma_0$ with the deck group $D$ generated by $\zeta$.  The points on $\Sigma_g$ which are fixed by a non-trivial power of $\zeta$ map to the branch points $\vB \subset \Sigma_0$.

 The mapping class group of $\Sigma_0$ relative to $\mathcal{B}$, denoted  $\Mod(\Sigma_0,\mathcal{B})$, consists of homotopy classes of orientation preserving homeomorphisms of $\Sigma_0$ where both homotopies and homeomorphisms preserve $\vB$.  On the other hand, homeomorphisms of $\Sigma_g$ need not preserve the points fixed by $\zeta$.

Let $\hat D$ be the image of the deck group $D$ in $\Mod(\Sigma_g)$.  Let $\SMod_p(\Sigma_g)$ be the subgroup of $\Mod(\Sigma_g)$ consisting of isotopy classes of fiber preserving homeomorphisms. Then $\SMod_p(\Sigma_g)$ is equal to the normalizer of $\hat D$ in $\Mod(\Sigma_g)$ \cite[Theorem 4]{BH}.

 Due to work of Birman and Hilden \cite{BH1,BH}, it is known that $\SMod(\Sigma_g)/\hat D$ is isomorphic to a finite index subgroup of $\Mod(\Sigma_0,\vB)$ provided $g > 1$.  We will call the finite index subgroup of $\Mod(\Sigma_0,\vB)$ the {\it liftable mapping class group}, denoted $\LMod_p(\Sigma_0,\vB)$.  The liftable mapping class group is exactly comprised of isotopy classes of homeomorphisms of $\Sigma_0$ that lift to homeomorphisms of $\Sigma_g$.  

\p{The hyperelliptic involution} The isomorphism $\SMod(\Sigma_g)/\hat D \cong \LMod_p(\Sigma_0,\vB)$ has been successfully expoited, most notably in the case where $\zeta$ is a hyperelliptic involution e.g. A'Campo \cite{acampo}, Arnol'd \cite{arnold}, Brendle-Margalit-Putman \cite{BMP}, Gries \cite{gries}, Hain \cite{hain}, Magnus-Peluso \cite{MP}, Morifuji \cite{morifuji}, Stukow \cite{stukow}.  Here $\SMod_p(\Sigma_g)$ is called the {\it hyperelliptic mapping class group}.  When $g=2$, the hyperelliptic mapping class group is equal to $\Mod(\Sigma_2)$.  Birman and Hilden used this fact to find the first presentation for $\Mod(\Sigma_2)$.  Bigelow and Budney proved that $\SMod_p(\Sigma_g)$ is linear \cite{BB} when $\zeta$ is a hyperelliptic involution.

One of the reasons the covering space induced by a hyperelliptic involution has been fertile ground for research is that in this case the liftable mapping class group $\LMod_p(\Sigma_0,\vB)$ equals $\Mod(\Sigma_0,\vB)$.  
In general $\LMod_p(\Sigma_0,\vB)$ is only finite index in $\Mod(\Sigma_0,\vB)$.  Although the finite index implies that $\LMod_p(\Sigma_0,\vB)$ enjoys many properties of $\Mod(\Sigma_0,\vB)$ such as finite presentability and linearity, it must be better understood in order to use the relationship $\LMod_p(\Sigma_0,\vB)\cong\SMod(\Sigma_g)/\hat{D}$ for explicit calculations.

\p{Cyclic branched covers of a sphere} Every finite cyclic branched covering space of a sphere can be modeled by a {\it superelliptic curve}, a plane curve with equation of the form $y^k = f(x)$ for some $f(x) \in \CC[x]$, $k \in \NN$.  Indeed, choose distinct points $a_1,\ldots,a_t \in \CC$.  Then a cyclic branched cover of the sphere can be modeled by an irreducible plane curve $C$ defined by
\[
y^k = (x - a_1)^{d_1} \cdots (x- a_t)^{d_t}
\]
where $1 \leq d_i \leq k-1$ for all $i$.  Let $\wt C$ be the normalization of the plane curve $C$.  Projection onto the $x$-axis gives a $k$-sheeted cyclic branched covering $ \wt C \to \PP^1$ branched at the roots of $f(x)$ and possibly at infinity. 

Removing the branch points $\vB \subset \PP^1$ and their preimages in $\wt C$, we obtain a cyclic (unbranched) covering space of $\PP^1\setminus \vB$.  By the Galois correspondence for covering spaces, this covering is determined by the kernel of a surjective homomorphism $\phi:\pi_1(\PP^1\sm\vB,x) \to \ZZ/k\ZZ$ for some point $x \in \PP^1 \sm \vB$.  Let $\gamma_i$ be a loop based at $x$ that runs counterclockwise around the branch point $a_i$.  Then $\phi(\gamma_i) \equiv d_i \mod k$.  Note that the irreducibility of $C$ implies the surjectivity of $\phi$.

\p{The family of balanced superelliptic covers}
In this paper we study a specific family of superelliptic curves, where
\begin{equation}\label{balanced_equation}
y^k = (x - a_1)(x-a_2)^{k-1} \cdots (x - a_{2n+1})(x - a_{2n+2})^{k-1}.
\end{equation}
There is no branching at infinity.  As $k$ and $n$ vary, we call the family of normalized curves {\it balanced superelliptic curves}.

Topologically, the balanced superelliptic curves describe a covering space as follows.  Fix integers $g,k \geq 2$ such that $k-1$ divides $g$.  Let $p_{g,k}:\Sigma_g \to \Sigma_0$ be a cyclic branched covering map of degree $k$ branched at $2n+2$ points, where $n = g/(k-1)$.  In this case, we will denote $\LMod_p(\Sigma_0,\vB)$ by $\LMod_{g,k}(\Sigma_0,\vB)$.  We will refer to the surface $\Sigma_g$ and the covering map $p_{g,k}$ together as a {\it balanced superelliptic cover.}  When $k=2$ we recover the case where the deck group is generated by a hyperelliptic involution.  The example where $g=4$ and $k=3$ is shown in Figure \ref{3foldcover}.

\p{Goals} The goals of this paper are to intiate the study of $\LMod_p(\Sigma_0,\vB)$ and $\SMod_p(\Sigma_g)$ in general, and to remove the restriction that $\LMod_p(\Sigma,\vB)$ is equal to $\Mod(\Sigma,\vB)$ in programs such as Brendle--Margalit--Putman's \cite{BMP} and McMullen's \cite{mcmullen}.  In the case where $\Sigma_g\rightarrow\Sigma_0$ is a degree $k$ balanced superelliptic cover, we call $\SMod_p(\Sigma_g)$ the {\it balanced superelliptic mapping class group} and denote it $\SMod_{g,k}(\Sigma_g)$.

We focus on the family of balanced superelliptic covers for a number of reasons.  First, when $k > 2$ it is no longer the case that $\LMod_{g,k}(\Sigma_0,\vB) = \Mod(\Sigma_0,\vB)$.  Therefore the balanced superelliptic covers provide a family of counterexamples to Lemma 5.1 of \cite{BH}, which is in error (see \cite{erratum} and \cite{GW} for a correction).  Second, the covers can be embedded in $\RR^3$ so that the deck group is generated by a rotation about the $z$-axis.  This picture should provide insight into the study of the balanced superelliptic mapping class group.  

McMullen \cite{mcmullen}, Venkataramana \cite{venkataramana}, Chen \cite{chen}, and others have studied a family of cyclic branched covering spaces of the sphere $p:\Sigma_g\to\Sigma_0$ that also generalize the cover induced by a hyperelliptic involution.  Their family arises from curves of the form
\[
y^k = (x - a_1)(x-a_2) \cdots (x - a_n).
\]
Note that there may be branching at infinity.  In their family, every homeomorphism of $\Sigma_0$ that fixes the point at infinity lifts to a homeomorphism of $\Sigma_g$.  When $k=2$ we recover the cover induced by a hyperelliptic involution.

\p{General cyclic covers} Although our focus is on the balanced superelliptic covers, the results in this paper could be generalized to any cyclic branched cover over the sphere.  Indeed, let $p:\Sigma_g \to \Sigma_0$ be a cyclic branched cover with branch points $\vB \subset \Sigma_0$.  Let $\widehat\Psi:\LMod_{p}(\Sigma_0,\vB) \to GL_{\abs{\vB} - 1}(\ZZ)$ be the homomorphism given by the action of $\LMod_{p}(\Sigma_0,\vB)$ on $H_1(\Sigma_0\sm \vB;\ZZ)$.  Then $\widehat\Psi(\LMod_p(\Sigma_0,\vB))$ is isomorphic to a subgroup of the symmetric group $S_{\abs{\vB}}$.  While calculating this subgroup of $S_{\abs{\vB}}$ is feasible in practice for a single cover or family of covers, we do not see a way to state an explicit general form.  If one were able to find a presentation for $\Psi(\LMod_p(\Sigma_0,\vB))$ in general, then the results of this paper could be generalized to all cyclic branched covers using the techniques developed within.  

\begin{figure}[t]
\begin{center}
\labellist\small\hair 2pt
       \pinlabel {$x$} by 1 0 at 24 50
    \pinlabel {$y$} by 1 0 at 5 60
      \pinlabel {$z$} by 0 0 at 25 35
        \endlabellist
\includegraphics[scale=1]{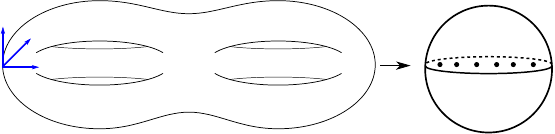}

\end{center}
\caption{The 3-fold cyclic branched covering space of $\Sigma_4$ and $\Sigma_4\to\Sigma_0$.}
\label{3foldcover}
\end{figure}

\subsection{Results}
Let $p:\Sigma_g \to \Sigma_0$ be the $k$-fold superelliptic covering space branched at $2n + 2$ points.  For $k > 2$ we compute the index $\left[\Mod(\Sigma_0,\vB):\LMod_{g,k}(\Sigma_0,\vB)\right] = \frac{(2n+2)!}{2((n+1)!)^2}$ in Scholium \ref{index}.  In fact, for a fixed number of branch points, the liftable mapping class is independent of the degree of the cover.  That is, for any integers $g_1,g_2$ and $k_1,k_2>2$ such that $k_i - 1$ divides $g_i$ and $g_1/(k_1 - 1) = g_2/(k_2 - 1)$ for $i = 1,2$, $\LMod_{g_1,k_1}(\Sigma_0,\vB) = \LMod_{g_2,k_2}(\Sigma_0,\vB)$.

The main technical result in the paper is an explicit presentation for \linebreak $\LMod_{g,k}(\Sigma_0,\vB)$ in Theorem \ref{maintheorem}.  This allows us to prove our main theorems.

\begin{thm}\label{abelianization}
Let $k \geq 3$.  Then
\[
H_1(\LMod_{g,k}(\Sigma_0,\vB);\ZZ) \cong \begin{cases}
\ZZ/2\ZZ \times \ZZ/2\ZZ \times \ZZ & \text{if $n$ is odd} \\
\ZZ/2\ZZ \times \ZZ & \text{if $n$ is even}.
\end{cases}
\]
\end{thm}

\begin{thm} \label{betti}
The abelianization of the balanced superelliptic mapping class group $H_1(\SMod_{g,k}(\Sigma_g);\ZZ)$ is an infinite non-cyclic abelian group.  Furthermore, the first Betti number of $\SMod_{g,k}(\Sigma_g)$ is 1.
\end{thm}


\subsection{Applications and future work}

In the family of covers where $p:\wt{\Sigma}\to\Sigma$ is a 3-fold, simple branched cover of the disk, Birman and Wajnryb found a presentation for $\LMod_p(\Sigma,\vB)$ \cite{BW}.  However, a 3-fold simple cover does not induce an isomorphism between $\LMod_p(\Sigma,\vB)$ and $\SMod_p(\wt{\Sigma})$ \cite{BE,winarski}.  In contrast, the balanced superelliptic covers we study do induce the Birman--Hilden isomorphism.  Therefore the presentation of $\LMod_{g,k}(\Sigma_0,\vB)$ can be used to find a presentation for $\SMod_{g,k}(\Sigma_g)$.

In particular, the generators of $\LMod_{g,k}(\Sigma_0,\vB)$ give us the generators of \linebreak $\SMod_{g,k}(\Sigma_g)$, which is an infinite index subgroup of $\Mod(\Sigma_g)$.

\begin{cor}\label{generators_smod}
Let $\Sigma_g$ be a surface of genus $g\geq 2$.  Let $p_{g,k}:\Sigma_g\rightarrow\Sigma_0$ be a balanced superelliptic cover of degree $k\geq 3$ with set of branch points $\vB=\vB(2n+2)$.  Choose lifts of each of the generators of $\LMod_{g,k}(\Sigma_0,\vB)$.  The subgroup $\SMod_{g,k}(\Sigma_g,\vB)$ is generated by these lifts and a generator of the deck group of $p_{g,k}$.
\end{cor}
\p{Generation by torsion elements} Stukow proved that the hyperelliptic mapping class group $\SMod_{g,2}(\Sigma_g)$ is generated by two torsion elements \cite{stukow1}.  We ask if there is an analogue for general $\SMod_{g,k}(\Sigma_g)$.
\begin{ques}
Can $\SMod_{g,k}(\Sigma_g)$ be generated by a small number of torsion elements?
\end{ques}

\p{Monodromy representation} Let $\Sigma_g$ be a genus $g$ surface and the map $\Sigma_g\rightarrow\Sigma_0$ be a cyclic branched cover.  Let $D$ be the deck group of the covering space and $\hat{D}$ the image of $D$ in $\Mod(\Sigma_g)$.  The mapping class group $\Mod(\Sigma_g)$ acts on $H_1(\Sigma_g,\mathbb{Z})$ and the action preserves the interesection form on $H_1(\Sigma_g,\mathbb{Z})$.  Thus the action induces a surjective representation $\rho:\Mod(\Sigma_g)\rightarrow \Sp(2g,\mathbb{Z})$.  Let $G$ be a group.  Let $C_G(H)$ denote the centralizer of a subgroup $H$ in $G$.  McMullen asks when $\rho(C_{\Mod(\Sigma_g)}(D))$ is finite index in $C_{\Sp(2g,\mathbb{Z})}(\rho(D))$ \cite{mcmullen}.  While McMullen looks at a different family of covering spaces than we do, our work could be used to extend his program to the family of balanced superelliptic curves.

\begin{ques}
What is the image $\rho(\SMod_{g,k}(\Sigma_g,\vB))$ in $\Sp(2g,\mathbb{Z})$?
\end{ques}

Since $\SMod_{g,k}(\Sigma_g)$ is the normalizer of $\hat{D}$ in $\Mod(\Sigma_g)$, we have an analogue of McMullen's question \cite{mcmullen}:
\begin{ques}
Let $p:\Sigma_g\rightarrow \Sigma_0$ be any cyclic branched cover of the sphere.  When is $\rho(\SMod_{p}(\Sigma_g))$ finite-index in the normalizer of $\rho(\hat{D})$?
\end{ques}

The generators for $\SMod_{g,k}(\Sigma_g)$ in corollary \ref{generators_smod} may be useful in answering this question for balanced superelliptic covers, and as noted above, it is possible to extend our techniques to other superelliptic covers.

\p{Outline of paper}  In section \ref{background}, we review the necessary combinatorial group theory and lifting properties for constructing our presentation.  In section \ref{the_covers}, we explicitly construct the family of balanced superelliptic covers, and we prove that $\LMod_{g,k}(\Sigma_0,\vB)$ is an extension of a subgroup $W_{2n+2}$ of the symmetric group $S_{2n+2}$ by the pure mapping class group $\PMod(\Sigma_0,\vB)$.  In section \ref{PMod_and_W} we find presentations for $\PMod(\Sigma_0,\vB)$ and $W_{2n+2}$ in the group extension.  We build the presentation for $\LMod_{g,k}(\Sigma_0,\vB)$ in section \ref{mainproof}.  Finally, we prove theorems \ref{abelianization} and \ref{betti} in section \ref{abelianization_section}.

\p{Acknowlegements} The authors would like to thank Joan Birman, Tara Brendle, Neil Fullarton, Mike Hilden, Lalit Jain, Dan Margalit, David McKinnon, Kevin Kordek and Doug Park for their comments and suggestions.  The authors would also like to thank the referee for suggestions for clarification, and Michael L\"onne for pointing out an error in an earlier version.  The second author would like to thank Doug Park's NSERC Discovery Grant for support to visit University of Waterloo.

\section{Preliminary definitions and lemmas}\label{background}

In this section, we survey the combinatorial group theory and algebraic topology results used later in the paper.
We first find a presentation of a group when given a short exact sequence of groups in section \ref{presentation_short_exact}.  We then use homological arguments to characterize the mapping classes that lift.  

\subsection{Group Presentations and Short Exact Sequences}\label{presentation_short_exact}

To obtain the presentation in section \ref{mainproof}, we use two well-known results concerning short exact sequences and group presentations.  

\begin{lem} \label{pres_quotient}
Let
\[
1 \lra K\overset{\alpha}{\longrightarrow} G\overset{\pi}{\lra} H\lra 1
\]
be a short exact sequence of groups.
Let $\langle S\mid R\rangle$ be a presentation for $G$ where each symbol $s \in S$ denotes a generator $g_s \in G$.  Let $K$ be normally generated by $\{k_\beta\} \subset K$ and for each $\beta$, let $w_\beta$ be a word in the symbols $S$ denoting $\alpha(k_\beta)$.  Then $H$ admits the presentation $\langle S \mid R \cup \{w_\beta\} \rangle$ where $s \in S$ denotes $\pi(g_s)$. \hfill
\end{lem}

A proof of Lemma \ref{pres_quotient} can be found in \cite[Section 2.1]{MKS} 

For Lemma \ref{pres_short_exact}, let \[
1 \lra K\overset{\alpha}{\longrightarrow} G\overset{\pi}{\lra} H\lra 1
\] be a short exact sequence of groups.  Let $K \cong \langle S_K \mid R_K \rangle$.  Let $t\in S_K$.  Assign the generator $k_t\in K$ to $t$.  Similarly, let $H \cong \langle S_H \mid R_H \rangle$.  Let $s\in S_H$.  Assign the generator $h_s\in H$ to $s$.

For each generator $h_s$ of $H$, choose an element $g_s \in G$ such that $\pi(g_s) = h_s$.  Then let $s \in S_H$, and let $\tilde s$ denote $g_s$.  Let $\wt S_H=\{\tilde s: s\in S_H\}$.  For each $k_t\in K$, let $\tilde t \in \wt S_K$ denote $\alpha(k_t) \in G$.  Let $\wt S_K = \{\tilde t: t \in S_K\}$.

Each word in $R_H$ can be written in the form  $s_1^{\epsilon_1}\cdots s_m^{\epsilon_m}$ with $s_i \in S_H$ and $\epsilon_i \in \{\pm 1\}$.  Let $r\in R_H$ be  $s_1^{\epsilon_1}\cdots s_m^{\epsilon_m}$.  Denote the word $\wt{s_1}^{\epsilon_1}\cdots \wt{s_m}^{\epsilon_m}$ in $\wt S_H$ by $\wt r$.  Then $\widetilde r$ is a word in $\wt S_H$ denoting some $g \in G$.  The element $g$ is such that $\pi(g) = \id_H$.  Since the sequence is exact, this means that $g \in \alpha(K)$.  Let $w_r$ be a word in $\wt S_K$ denoting $g$ and define the set of words
\[
R_1 := \{\tilde r w_r^{-1} : r \in R_H\}.
\]

Since $\alpha(K)$ is normal in $G$, for every $k_t\in K$ and $g_s\in G$, the element $g_s\alpha(k_t)g_s^{-1}\in\alpha(K)$.  Let $v_{s,t}$ be a word in $\wt S_K$ that denotes $g_s\alpha(k_t)g_s^{-1}$.  Define the set of words
\[
R_2 := \{\tilde s \tilde t \tilde s^{-1} v_{s,t}^{-1} : \tilde t \in \wt S_K, \tilde s \in \wt S_H\}.
\]

Finally, let $\wt R_K := \{\tilde r: r \in R_K\}$ where $\tilde r$ is the word in $\wt S_K$ obtained by replacing every symbol $t$ by $\tilde t$ in the same way as in the definition of $R_1$.

\begin{lem} \label{pres_short_exact}
Let
\[
1 \lra K\overset{\alpha}{\longrightarrow} G\overset{\pi}{\lra} H\lra 1
\]
be a short exact sequence of groups.  Then $G$ admits the presentation
\[
G \cong \langle \wt S_K \cup \wt S_H \mid R_1 \cup R_2 \cup \wt R_K\rangle.
\] where $\wt S_K, \wt S_H, R_1, R_2,$ and $\wt R_K$ are defined as above. \hfill
\end{lem}

A proof is left to the reader.

\subsection{Lifting mapping classes}

Our goal is to characterize which mapping classes in $\Mod(\Sigma_{0},\vB)$ belong to $\LMod_{g,k}(\Sigma_{0},\vB)$.  Because all homotopies of $\Sigma_{0}$ lift to homotopies of $\Sigma_g$, it is sufficient to determine which homeomorphisms of $\Sigma_{0}$ lift to homeomorphisms of $\Sigma_g$.  In \ref{liftingcurves} we characterize curves in $\Sigma_0$ that lift to closed curves in $\Sigma_g$.  In \ref{liftinghomeo} we characterize homeomorphisms of $\Sigma_0$ that lift to homeomorphisms of $\Sigma_g$.

\subsubsection{Lifting Curves}\label{liftingcurves}
 Throughout this section we will work in generality.  Let $\widetilde{X}$ be a path connected topological space. 

 Let $p: \wt X\to X$ be an unbranched covering space.  Let $c:S^1\to X$ be a curve in $X$.  Recall that $c$ {\it lifts} if there exists $\tilde c: S^1 \to \wt X$ such that $p\tilde c = c$.

Let $p:\wt X \to X$ be an abelian covering space with deck group $D$. Fix a base point $x_0\in X$.  There is a one-to-one correspondence between regular covering spaces of $X$ and normal subgroups of $\pi_1(X,x_0)$.  The covering space $p:\wt X\to X$ corresponds to the kernel of a surjective homomorphism $\varphi:\pi_1(X,x_0) \to D$.   Let $\Phi:\pi_1(X,x_0) \to H_1(X;\ZZ)$ be the Hurewicz homomorphism.  Since $D$ is abelian, there exists a homomorphism $\ol \varphi:H_1(X;\ZZ) \to D$ such that $\varphi = \ol \varphi \Phi$.

Conversely, given a homomorphism $\ol \varphi:H_1(X;\ZZ) \to D$, we can define a homomorphism $\varphi:\pi_1(X,x_0) \to D$ by setting $\varphi = \ol \varphi \Phi$.  Since $\ker(\varphi)$ is a normal subgroup of $\pi_1(X,x_0)$, it determines a regular cover.  So, for a regular abelian cover we will call the homomorphism $\ol \varphi:H_1(X;\ZZ) \to D$ the {\it defining homomorphism} of the cover.  Note that this homomorphism is well defined up to an automorphism of $D$.

Unwrapping these definitions we get the following lemma.

\begin{lem} \label{curve_lift_homology_kernel}
Let $p:\wt X \to X$ be a regular abelian cover with deck group $D$, and let $\ol\varphi:H_1(X;\ZZ) \to D$ be the defining homomorphism.  A curve $c:S^1 \to X$ lifts if and only if $[c] \in \ker \ol \varphi < H_1(X;\ZZ)$. \hfill
\end{lem}

\subsubsection{Lifting homeomorphisms}\label{liftinghomeo}
Let $\widetilde{X}$ be a path connected topological space.  Let $p:\widetilde{X}\rightarrow X$ be a finite-sheeted covering map.

A homeomorphism $f:X \to X$ {\it lifts} if there exists a homeomorphism $\tilde f:\wt X \to \wt X$ such that $p\tilde f = fp$

Let $f$ be a homeomorphism of $X$ and let $f_*$ be the induced map on $H_1(X,\mathbb{Z})$.

\begin{lem} \label{homeo_lift}
Let $p:\wt X \to X$ be a regular abelian cover, with $X,\wt X$ path connected.  Let $D$ be the deck group and let $\ol \varphi:H_1(X) \to D$ be the defining homomorphism.  Then a homeomorphism $f:X \to X$ lifts if and only if the induced map on homology $f_*$ satisfies $f_*(\ker \ol \varphi) = \ker \ol \varphi$. \hfill
\end{lem}

A well-known corollary follows immediately:
\begin{cor}\label{lifting_criterion}
Let $\wt{X}$ be a path connected topological space.  Let $p:\wt X\to X$ be an abelian cover.  A homeomorphism $f:X\to X$ lifts if and only if for all curves $c$ that lift, $f(c)$ also lifts. 
\end{cor}

\p{Surfaces} Let $\Sigma_g$ be a closed surface and let $\vB(m) \subset \Sigma_g$ be a set of $m$ marked points in $\Sigma_g$.  Let $\Sigma_{g,m}^{\circ} = \Sigma_g \sm \vB(m)$.  If the number of punctures $m$ is either clear from context or irrelevent, we will write $\Sigma_g^{\circ}$ to denote a surface with punctures.  

\subsection{The group extension}\label{extension}

Let $\PMod(\Sigma_{0,m}^{\circ})$ denote the pure mapping class group of $\Sigma_{0,m}^{\circ}$. The pure mapping class group is the subgroup of $\Mod(\Sigma_{0,m}^{\circ})$ that fixes each of the punctures.  Let $S_m$ be the symmetric group on $m$ elements.  There is an exact sequence:
\begin{equation}\label{permuting_branch_points} 1\to\PMod(\Sigma_{0,m}^{\circ})\to\Mod(\Sigma_{0,m}^{\circ})\to S_m\to 1.\end{equation}

Let $p:\Sigma_g\to\Sigma_0$ be a finite branched covering space with set of $m$ branch points $\vB(m)$.  Our goal is to find an sequence analogous to (\ref{permuting_branch_points}) for \linebreak $\LMod_p(\Sigma_0,\vB(m))$.  

\p{Action on homology} The first homology group $H_1(\Sigma_{0,m}^\circ;\ZZ)$ is isomorphic to $\ZZ^{\oplus m-1}$, and a basis can be chosen as follows.  Number the punctures $1,\ldots,m$, and let $x_i$ be the homology class of curve on $\Sigma_{0,m}^\circ$ surrounding the $i$th puncture, oriented counterclockwise around the puncture. Then $\{x_1,\ldots,x_{m-1}\} \subset H_1(\Sigma_{0,m}^\circ)$ forms a basis.

Let $\Psi_m:\Mod(\Sigma_{0,m}^\circ) \to \GL_{m-1}(\ZZ)$ be the homomorphism given by the action of $\Mod(\Sigma_{0,m}^\circ)$ on $H_1(\Sigma_{0,m}^\circ;\ZZ)$.  Since each basis element is supported on a neighborhood of a puncture, any element of the pure mapping class group will act trivially on homology.  Conversely, any homeomorphism which induces a non-trivial permutation on the punctures will permute homology classes of loops surrounding the punctures.

From this discussion we see that the kernel of $\Psi_m$ is equal to the pure mapping class group $\PMod(\Sigma_{0,m}^\circ)$, and the image of $\Psi_m$ is isomorphic to the symmetric group $S_m$.  Indeed, if $f$ is a homeomorphism of $\Sigma_{0,m}^\circ$, $\Psi_m([f])$ is the permutation induced on the $m$ punctures.  We can now conclude that the short exact sequence (\ref{permuting_branch_points}) above is obtained from the action of $\Mod(\Sigma_{0,m}^\circ)$ on $H_1(\Sigma_{0,m}^\circ;\ZZ)$.

\p{Punctures and marked points} Our lifting criteria above can only be applied to unbranched covering spaces.  However we ultimately want a presentation for $\LMod_p(\Sigma_0,\vB)$, where $\Sigma_0$ is a surface with branch points $\vB$.  

To resolve the distinction between punctures and branch points, let $p:\wt{\Sigma}\to \Sigma$ be a branched covering space of surfaces with set of branch points $\vB\subset\Sigma$.  As above, it may be necessary to remove the branch points in $\Sigma$ to obtain the punctured surface $\Sigma^{\circ}=\Sigma\setminus\vB$.  We then must also remove the preimages of the branch points in $\wt {\Sigma}$ to obtain the punctured surface $\wt{\Sigma}^{\circ}=\wt{\Sigma}\setminus p^{-1}(\vB)$.  Let $p\big|_{\widetilde{\Sigma}^\circ}:\widetilde{\Sigma}^{\circ}\to \Sigma^{\circ}$ be an unbranched covering map.  We use the $\Sigma^{\circ}$ notation specifically when we work with a surface with branch points removed (or $\Sigma_{0,m}^{\circ}$ when we need to specify the number of punctures). There is an inclusion map $\iota:\Sigma^{\circ}\to\Sigma$ where the punctures of $\Sigma^{\circ}$ are filled in with marked points.  These marked points exactly comprise the set of branch points $\vB$ of the cover.  Then $\Mod(\Sigma^{\circ})$ is isomorphic to $\Mod(\Sigma,\vB)$ because homeomorphisms and homotopies of $\Sigma_{0}^{\circ}$ must fix the set of punctures and $\Mod(\Sigma,\vB)$ must fix the set of branch points.  On the other hand, in the inclusion map $\wt{\iota}:\wt{\Sigma}^{\circ}\to\wt{\Sigma}$ the punctures are filled in with non-marked points. Then $\Mod(\wt{\Sigma})$ and $\Mod(\wt{\Sigma}^{\circ})$ are not isomorphic because the set of points $p^{-1}(\vB)$ are not treated as marked points in $\Mod(\wt{\Sigma})$.

Work of Birman and Hilden \cite{BH1,BH} gives an isomorphism between \linebreak $\LMod_p(\Sigma,\vB)$ and a subgroup of $\Mod(\wt\Sigma)$ modulo the homotopy classes of the deck transformations.  The group $\Mod(\wt\Sigma)$ need not stabilize the set $p^{-1}(\vB)$ in their work.

\section{The balanced superelliptic covers}\label{the_covers}
\subsection{The construction}\label{construction}
Choose a pair of integers $g,k \geq 2$ such that $k-1$ divides $g$, and let $n = g/(k-1)$.  Embed $\Sigma_g$, a surface of genus $g$, in $\RR^3$ so it is invariant under a rotation by $2\pi/k$ about the $z$-axis as we describe below. 

The intersection of $\Sigma_g$ with the plane $z = a$ is:
\begin{itemize}
\item Empty for $a< 0$ and $a > 2n+1$ 
\item A point at the origin for $a = 0$ and $a = 2n+1$
\item Homeomorphic to a circle for $2m < a < 2m+1$ with $m \in \{0,\ldots,n\}$
\item A rose with $k$ petals for $a \in \{1,\ldots,2n\}$
\item $k$ disjoint simple closed curves invariant under a rotation of $2\pi/k$ about the $z$-axis for $2m-1 < a< 2m$ with $m \in \{1,\ldots,n\}$.  In the special case $a = 2m-1/2$, put polar coordinates $(r,\theta)$ on the plane $z = 2m-1/2$.  Then we have $k$ disjoint circles with centers on the rays $\theta = 2\pi d/k$, $d \in \{0,\ldots, k-1\}$.
\end{itemize}
See Figure \ref{3foldcover} for the embedding when $g = 4$ and $k = 3$.

Consider a homeomorphism $\zeta:\Sigma_g \to \Sigma_g$ of order $k$ given by rotation about the $z$-axis by $2\pi/k$.  The homeomorphism $\zeta$ fixes $2n+2$ points, which are the points of intersection of $\Sigma_g$ with the $z$-axis.  Define an equivalence relation on $\Sigma_g$ given by $x \sim y$ if and only if $\zeta^q(x) = y$ for some $q$.  The resulting surface $\Sigma_g/\sim$ is homeomorphic to a closed sphere $\Sigma_0$.  The quotient map $p_{g,k}:\Sigma_g \to \Sigma_0$ is a $k$-fold cyclic branched covering map with $2n+2$ branch points, which are the images of the points fixed by $\zeta$.    The deck group of $p_{g,k}$ is a cyclic group of order $k$ generated by $\zeta$.  
When $k = 2$, $\zeta$ is a hyperelliptic involution.

\p{An important collection of arcs}  Fix a pair of integers $g,k \geq 2$ such that $k-1\mid g$, and consider the surface $\Sigma_g$ embedded in $\RR^3$ as described above.  Using cylindrical coordinates in $\RR^3$, let $P_{\theta_0} = \{(r,\theta_0,z) \in \RR^3 : r\geq 0\}$ be a closed half plane.  The intersection of $\Sigma_g$ and $P_{\pi/k}$ is a collection of $n+1$ arcs where $n = g/(k-1)$.  Call these arcs $\beta_1,\ldots,\beta_{n+1}$.  For each arc $\beta_i:[0,1] \to \Sigma_g$, orient it so that $\beta_i(0) = (0,0,2i-2)$ and $\beta_i(1) = (0,0,2i-1)$ in $\RR^3$. Number the endpoints $1,\cdots, 2n+2$ in order of increasing $z$ value and fix the numbering for the remainder of the paper.  

Consider the balanced superelliptic covering map $p_{g,k}$ as defined above. For each $i$ with $1\leq i\leq n+1$, let  $\alpha_i= p_{g,k}\beta_i$.  Each $\alpha_i$ is an arc $\alpha_i:[0,1]\to\Sigma_0$.  The endpoints of $\alpha_i$ are in the set of branch points $\vB(2n+2)\subset\Sigma_0$.  Let $\alpha$ be the union of the arcs $\alpha_i$ in $\Sigma_0$.  Let $[\alpha]$ denote the relative homology class of $\alpha$ in $H_1(\Sigma_0,\vB;\ZZ)$.  The class $[\alpha]$ is calculated $\sum_{i=1}^{n+1} [\alpha_i]$.  Figure \ref{arcs} shows the embeddings of the arcs $\beta_1,\beta_2,\beta_3 \in\Sigma_4$ and $\alpha_1,\alpha_2,\alpha_3\in\Sigma_0$ for the 3-fold balanced superelliptic cover of $\Sigma_4$ over $\Sigma_0$. 

\begin{figure}[t]
\begin{center}
\labellist\small\hair 2.5pt
       \pinlabel {$\beta_1$} by 2 0 at 24 40
    \pinlabel {$\beta_2$} by 2 0 at 102 40
      \pinlabel {$\beta_3$} by 0 0 at 170 40
             \pinlabel {$\alpha_1$} by 1 0 at 224 40
    \pinlabel {$\alpha_2$} by 1 0 at 240 40
      \pinlabel {$\alpha_3$} by 1 0 at 259 40
        \endlabellist
\includegraphics[scale=1]{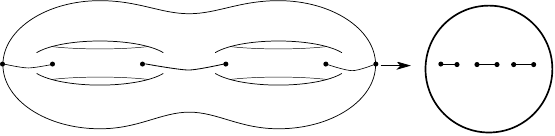}

\end{center}
\caption{The arcs $\beta_1,\beta_2,\beta_3\in\Sigma_4$ and the arcs $\alpha_1,\alpha_2,\alpha_3\in\Sigma_0$.}
\label{arcs}
\end{figure}
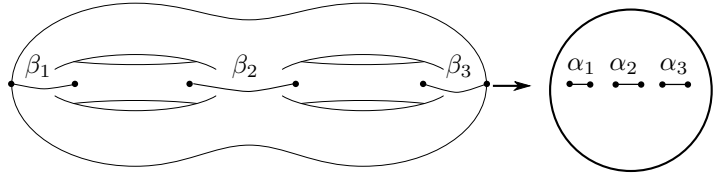

\subsection{A lifting criterion for superelliptic covers}\label{resolution}
The goal of this section is to prove Lemma \ref{curve_lifting_criterion}.  

\p{An intersection form for punctured surfaces}
In Lemma \ref{homology_intersection_homomorphism}, we abuse notation and identify curves in $\Sigma_{0,m}^\circ$ with their image in $\Sigma_0$ under the inclusion $\Sigma_{0,m}^\circ \to \Sigma_0$.

\begin{lem} \label{homology_intersection_homomorphism}
Let $\Sigma_g$ be a closed surface and $\vB(m)$ a set of $m$ points in $\Sigma_g$.  There exists a homomorphism 
\[
\bullet: H_1(\Sigma_{g,m}^\circ;\ZZ) \times H_1(\Sigma_g,\vB(m);\ZZ) \to \ZZ
\]
given by $c \bullet a = \hat i(c, a)$ where $c$ is a homotopy class of curves in $\Sigma_{g,m}$, $a$ is a homotopy class curves or arcs in $(\Sigma_g,\vB)$, and $\hat i(c,a)$ is the algebraic intersection of $c$ and $a$. 
\end{lem}
This is a well-known result and a proof can be found in the appendix of \cite{GM}, for example. 

We need the following combinatorial lemma.

\begin{lem}\label{digraphwalk}
Let $G$ be the weighted digraph
\begin{center}
\begin{tikzpicture}[>=stealth, scale = 1.5]
	\draw[fill=black] (0,0) circle (.05cm);
	\draw[rotate=-112.5,shift = {(0,1)},fill=black] (0,0) circle (.05cm);
	\draw[rotate=-112.5,shift = {(0,1)},rotate=-45,shift={(0,1)},fill=black] (0,0) circle (.05cm);
	\draw[rotate=112.5,shift = {(0,1)},fill=black] (0,0) circle (.05cm);
	\draw[rotate=112.5,shift = {(0,1)},rotate=45,shift={(0,1)},fill=black] (0,0) circle (.05cm);

	\draw[rotate=-112.5,scale=.5,decoration={markings,mark=at position .55 with {\arrow[scale=2]{<}}}, postaction={decorate}] (0,0) ..node[below]{$-1$} controls (.3,.7) and (.3, 1.3) .. (0,2);
	\draw[rotate=-112.5,shift = {(0,1)}, rotate=-45, scale=.5,decoration={markings,mark=at position .55 with {\arrow[scale=2]{<}}}, postaction={decorate}] (0,0) ..node[below left]{$-1$} controls (.3,.7) and (.3, 1.3) .. (0,2);
	\draw[rotate=112.5,shift = {(0,1)}, rotate=180, scale=.5,decoration={markings,mark=at position .55 with {\arrow[scale=2]{<}}}, postaction={decorate}] (0,0) ..node[below]{$-1$} controls (.3,.7) and (.3, 1.3) .. (0,2);
	\draw[rotate=112.5,shift = {(0,1)}, rotate=45,shift={(0,1)},rotate=180, scale=.5,decoration={markings,mark=at position .55 with {\arrow[scale=2]{<}}}, postaction={decorate}] (0,0) ..node[below right]{$-1$} controls (.3,.7) and (.3, 1.3) .. (0,2);

	\draw[rotate=-112.5,scale=.5,decoration={markings,mark=at position .55 with {\arrow[scale=2]{>}}}, postaction={decorate}] (0,0) ..node[above]{$+1$} controls (-.3,.7) and (-.3, 1.3) .. (0,2);
	\draw[rotate=-112.5,shift = {(0,1)}, rotate=-45, scale=.5,decoration={markings,mark=at position .55 with {\arrow[scale=2]{>}}}, postaction={decorate}] (0,0) ..node[above right]{$+1$} controls (-.3,.7) and (-.3, 1.3) .. (0,2);
	\draw[rotate=112.5,shift = {(0,1)}, rotate=180, scale=.5,decoration={markings,mark=at position .55 with {\arrow[scale=2]{>}}}, postaction={decorate}] (0,0) ..node[above]{$+1$} controls (-.3,.7) and (-.3, 1.3) .. (0,2);
	\draw[rotate=112.5,shift = {(0,1)}, rotate=45,shift={(0,1)},rotate=180, scale=.5,decoration={markings,mark=at position .55 with {\arrow[scale=2]{>}}}, postaction={decorate}] (0,0) ..node[above left]{$+1$} controls (-.3,.7) and (-.3, 1.3) .. (0,2);

	\draw[dashed, rotate=-112.5,shift = {(0,1)}, rotate=-45, shift = {(0,1)}, rotate = -45, scale=.5, decoration={markings,mark=at position .55 with {\arrow[scale=2]{>}}}, postaction={decorate}] (0,0) ..node[below right]{$+1$} controls (-.3,.7) and (-.3, 1.3) .. (0,2);
	\draw[dashed, rotate=-112.5,shift = {(0,1)}, rotate=-45, shift = {(0,1)}, rotate = -45, scale=.5, decoration={markings,mark=at position .55 with {\arrow[scale=2]{<}}}, postaction={decorate}] (0,0) ..node[left]{$-1$} controls (.3,.7) and (.3, 1.3) .. (0,2);
	\draw[dashed, rotate=112.5,shift = {(0,1)}, rotate=45, shift = {(0,1)}, rotate = 45, shift = {(0,1)}, rotate = 180, scale=.5, decoration={markings,mark=at position .55 with {\arrow[scale=2]{>}}}, postaction={decorate}] (0,0) ..node[below left]{$+1$} controls (-.3,.7) and (-.3, 1.3) .. (0,2);
	\draw[dashed, rotate=112.5,shift = {(0,1)}, rotate=45, shift = {(0,1)}, rotate = 45, shift = {(0,1)}, rotate = 180, scale=.5, decoration={markings,mark=at position .55 with {\arrow[scale=2]{<}}}, postaction={decorate}] (0,0) ..node[right]{$-1$} controls (.3,.7) and (.3, 1.3) .. (0,2);

	\def\x{.35};
	\def\y{-2.3};
	\begin{scope}[shift={(0,\y)}]
	\draw[fill=black] (0,0) circle (0.025cm);
	\draw[fill=black] (\x,0) circle (0.025cm);
	\draw[fill=black] (-\x,0) circle (0.025cm);
	\end{scope}

	\node[above] at (0,0) {$S_0$};
	\begin{scope}[rotate=-112.5,shift={(0,1)}]
	\node[above right] at (0,0) {$S_1$};
	\begin{scope}[rotate=-45,shift={(0,1)}]
	\node[right] at (0,0) {$S_2$};
	\end{scope}
	\end{scope}
	\begin{scope}[rotate=112.5,shift={(0,1)}]
	\node[above left] at (0,0) {$S_{k-1}$};
	\begin{scope}[rotate=45,shift={(0,1)}]
	\node[left] at (0,0) {$S_{k-2}$};
	\end{scope}
	\end{scope}
\end{tikzpicture}
\end{center}
and let $\Gamma$ be a finite walk on $G$ beginning at $S_0$.  Define the weight of $\Gamma$, which we denote $w(\Gamma)$, as the sum of the weights of the edges traversed in the walk.  Then $\Gamma$ terminates at $S_i$, if and only if $w(\Gamma) \equiv i \mod k$.
\end{lem}
\begin{proof}
Induct on the length of the walk.
\end{proof}

In order to apply this lemma, we use the union of arcs $\alpha$ in $\Sigma_0$ defined in section \ref{construction} and their preimages in $\Sigma_g$.

The full preimage $p_{g,k}^{-1}(\alpha)$ is a collection of $k(n+1)$ oriented arcs in $\Sigma_g$ and we will denote the union by  $\wt\alpha$.  The union of arcs $\wt{\alpha}$ consists of the orbits $\beta_i$ under the action of the deck group of $p_{g,k}$.  

The surface $\Sigma_g \sm \{\wt \alpha\}$ is a union of $k$ subsurfaces of $\Sigma_g$.  The subsurfaces are cyclically permuted by the action of the deck group.  Label one of these connected subsurfaces $R_0$, and for each $\ell \in \{1,\ldots,k-1\}$ label $\zeta^\ell(R_0)$ by $R_{\ell}$.  We will refer to the embedding of each $R_{\ell}$ in $\Sigma_g$ as a {\it region} in $\Sigma_g$.  

Consider a curve $\wt\gamma$ in $\Sigma_g$ that does not contain any of the $2n+2$ points on the $z$-axes and that intersects  $\wt \alpha$ transversely.  Choose an orientation for $\wt\gamma$.

Denote the algebraic intersection of $\wt\gamma$ and $\wt{\alpha}$ by $\hat{i}(\wt\gamma,\wt{\alpha})$.   Homotopy preserves algebraic intersection, so the algebraic intersection is well defined for all representatives within the class.  The orientation of $\wt{\alpha}$ is consistent with respect to the covering map $p_{g,k}$.

Consider a parameterization $\wt\gamma:\left[0,1\right]\to \Sigma_g$ with $\wt\gamma(0)=\wt\gamma(1)$. Let $t_0\in\left[0,1\right]$ be a value such that $\wt\gamma(t_0)\in \wt\gamma\cap\wt{\alpha}$.  Fix $\epsilon>0$ such that  $\wt\gamma(t_0-\epsilon,t_0+\epsilon)\cap\wt{\alpha}=t_0$.  Then either $\wt\gamma(t_0-\epsilon)\in R_{(i\mod k)}$ and $\wt\gamma(t_0+\epsilon)\in R_{(i+1\mod k)}$ for some $0\leq i\leq k-1$ or $\wt\gamma(t_0-\epsilon)\in R_{(i\mod k)}$ and $\wt\gamma(t_0+\epsilon)\in R_{(i-1\mod k)}$ for some $1\leq i\leq k$. All intersections of $\wt\gamma$ and $\wt{\alpha}$ where the index of $R_i$ increases modulo $k$ will have the same sign of intersection.  All intersections of $\wt\gamma$ and $\wt{\alpha}$ where the index of $R_i$ decreases modulo $k$ will have sign of intersection opposite to those where the index of $R_i$ increases.

\begin{lem}\label{lift_necessary}
Let $\ol{p}_{g,k}:\Sigma_g^{\circ}\to\Sigma_0^{\circ}$ be an unbranched balanced superelliptic covering map of degree $k$.  If a curve $\gamma$ in $\Sigma_0^\circ$ lifts, then $\hat i(\gamma,\alpha) \equiv 0 \mod k$.
\end{lem}
\begin{proof}
Consider the regions $R_0,\cdots,R_{k-1}\subset (\Sigma_g\setminus \wt{\alpha})$ as above.  For each $R_j\in\Sigma_g$, there is a corresponding embedding of $R_j\setminus p_{g,k}^{-1}(\vB)$ in $\Sigma_{g}^{\circ}$.  We will also denote the embeddings of  $R_j\setminus p_{g,k}^{-1}(\vB)$ in $\Sigma_{g}^{\circ}$ by $R_j$ as it will be clear from context when we are referring to punctured regions.

Let $\gamma$ be a curve in $\Sigma_{0}^{\circ}$ that lifts to a multicurve $\wt \gamma$ in $\Sigma_g^{\circ}$.  The multicurve $\wt{\gamma}$ has $k$ components in $\Sigma_g^{\circ}$.   Each component of $\wt\gamma$ is a map $\left[0,1\right]/\{0,1\}\to\Sigma_g^{\circ}$.  Let $\wt{\gamma}_i$ denote the component of $\wt{\gamma}$ such that $\wt{\gamma}_i(0)=\wt{\gamma}_i(1)\in R_i$. 

By compactness, $\abs{\wt\gamma \cap \wt\alpha} < \infty$.  We may assume that $\wt\gamma$ is transverse to $\wt \alpha$. Let $\wt x\in\wt{\gamma}\cap\wt{\alpha}$.  Since the action of the deck group is transitive, the orbit of $\wt x$ is of order $k$.  Indeed, the orbit of $\wt x$ is exactly $\bar{p}^{-1}_{g,k}(\bar{p}_{g,k}(\wt x))$. The signs of intersection of all points in the orbit of $\wt x$ are equal. Thus all components of $\wt{\gamma}$ have the same algebraic and geometric intersections with $\wt\alpha$.  Therefore $\wh{i}(\gamma,\alpha)=\wh{i}(\wt{\gamma}_i,\wt\alpha)$ for any component $\wt\gamma_i$ of $\wt\gamma$.  

Let $G$ be the weighted digraph as in lemma $\ref{digraphwalk}$.  Let $S_i$ be the vertex in $G$ corresponding to the region $R_i$ in $\Sigma_g$.

We now construct a walk $\Gamma_{\wt{\gamma}_i}$ in $G$ corresponding to each $\wt\gamma_i$.  The walk $\Gamma_{\wt{\gamma}_i}$ begins at the vertex $S_i$.   If $\wt{\gamma}_i\cap\wt{\alpha}$ is empty, the point $S_i$ is the entire path $G$.  

If $\wt{\gamma}_i\cap\wt{\alpha}\neq\emptyset$, let $\{t_j\}\subset\left[0,1\right]$ be the set of values such that $\wt{\gamma}_i(t_j)\in\wt{\gamma}_i\cap\wt{\alpha}$ and $t_j<t_{j+1}$.  Choose a value $\epsilon>0$ so that $\wt{\gamma}_i(t_j-\epsilon,t_j+\epsilon)\cap\wt{\alpha}=t_j$ for each $t_j$.  We construct the walk $\Gamma_{\wt{\gamma}_i}$ by adding an edge and a vertex for each $t_j$, in the order of increasing $j$.     The vertices will be those corresponding to the regions containing the elements $\wt{\gamma}_i(t_j+\epsilon)$ for each $j$.  Add the edge corresponding to $t_\ell$, which connects the vertex corresponding to the region containing $\wt{\gamma}_i(t_\ell-\epsilon)$ to the vertex corresponding to the region containing $\wt{\gamma}_i(t_\ell+\epsilon)$.  

For each component $\wt{\gamma}_i$, the walk $\Gamma_{\wt\gamma_i}$ begins and terminates at $S_i$.  By Lemma \ref{digraphwalk}, $\wh{i}(\wt{\gamma}_i,\wt{\alpha})\equiv 0\mod k$. Then by the discussion above, $\wh{i}(\gamma,\alpha)\equiv 0\mod k$ as well.
\end{proof}

We are now ready to prove Lemma \ref{curve_lifting_criterion}, which is Lemma \ref{lift_necessary} and its converse.

\begin{lem}[A lifting criterion for curves] \label{curve_lifting_criterion}
Let $\ol p_{g,k}:\Sigma_g^\circ \to \Sigma_0^\circ$ be the unbranched balanced superelliptic covering space.  Let $\gamma$ be a curve on $\Sigma_0^\circ$.  Then $[\gamma] \in \ker(\ol\varphi)$ if and only if $\hat i(\gamma,\alpha) \equiv 0 \mod k$.
\end{lem}

We note that an analogue of Lemma \ref{curve_lifting_criterion} is true for all cyclic branched covers of the sphere, but the collection of arcs $\alpha$ is specific to the balanced superelliptic covers.

\begin{proof}
Let $\hat i (-,\alpha) : H_1(\Sigma_0^\circ;\ZZ) \to \ZZ$ be the homomorphism from Lemma \ref{homology_intersection_homomorphism}, and let $\pi:\ZZ \to \ZZ/k\ZZ$ be the natural projection map.  Let $\phi = \pi \circ\hat i(-,\alpha) : H_1(\Sigma_0^\circ;\ZZ) \to \ZZ/k\ZZ$.  The homomorphism $\phi$ is surjective since there is a curve $\gamma$ such that $\hat i(\gamma,\alpha) = 1$.

Let $\ol \varphi:H_1(\Sigma_0^\circ;\ZZ) \to \ZZ/k\ZZ$ be the defining homomorphism for the unbranched balanced superelliptic cover.  By Lemma \ref{curve_lift_homology_kernel}, $\ker(\ol \varphi) = \{[\gamma] \in H_1(\Sigma_0^\circ;\ZZ) : \gamma \text{ lifts}\}$.  Lemma \ref{lift_necessary} shows that $\ker(\phi) \subseteq \ker(\ol \varphi)$.  However, these are both index $k$ subgroups of $H_1(\Sigma_0^\circ;\ZZ)$, so $\ker(\phi) = \ker(\ol \varphi)$.

By the definition of $\phi$, $\ker(\phi) = \{[\gamma] \in H_1(\Sigma_0^\circ;\ZZ) : \hat i (\gamma,\alpha) \equiv 0 \mod k\}$.  This completes the proof.
\end{proof}

Recall from Lemma \ref{curve_lift_homology_kernel} that a curve $\gamma$ lifts if and only if $[\gamma] \in \ker(\ol\varphi)$.  This together with Lemma \ref{curve_lifting_criterion} allows us to decide whether or not a curve lifts simply by computing its algebraic intersection number with the collection of arcs $\alpha$.

\begin{lem}\label{curve_lift}
Let $\ol p_{g,k}$ be an unbranched balanced superelliptic covering map of degree $k$.  Number the punctures of $\Sigma_0^{\circ}$ from 1 to $2n+2$ as in section \ref{construction}.  Let $x_j$ be the homology class of curves surrounding the $j$th puncture of $\Sigma_0^{\circ}$ for $1\leq j\leq 2n+1$ and oriented counterclockwise. The set $\{x_1,\ldots,x_{2n+1}\}$ forms a basis for $H_1(\Sigma_0^\circ;\ZZ)$.  Let $\gamma$ be a curve in $\Sigma_0^{\circ}$ with $[\gamma] = (\gamma_1,\ldots,\gamma_{2n+1}) \in H_1(\Sigma_0^\circ;\ZZ)$ with respect to this basis.  Then $\gamma$ lifts if and only if $$\sum_{i=1}^{2n+1} (-1)^{i+1}\gamma_i \equiv 0 \mod k.$$
\end{lem}

Note that Lemma \ref{curve_lift} is the key lemma that distinguishes the balanced superelliptic covers from other superelliptic covers.  In order to use our methods for other families of superelliptic covers, one must characterize the curves that lift, as we do here for the balanced superelliptic covers.
\begin{proof}
Let $\alpha$ be the collection of arcs defined above and observe that
\[
\hat i(x_j,\alpha) = \begin{cases}
1 & \text{if $j$ is odd} \\
-1 & \text{if $j$ is even}.
\end{cases}
\]
Then $\hat i(\gamma,\alpha) = \sum_{i=1}^{2n+1} (-1)^{i+1} \gamma_i$.  Combining this with Lemma \ref{curve_lifting_criterion} completes the proof.
\end{proof}

This lemma shows that the family of balanced superelliptic covers are modeled by plane curves defined by equation (\ref{balanced_equation}).

\subsection{The exact sequence for $\LMod_p(\Sigma_0,\vB)$}\label{exact_LMod}  Let $p:\wt\Sigma\to\Sigma_0$ be a finite cyclic branched cover of the sphere.  Let $\vB(m)$ be the set of $m$ branch points of the covering space.  Recall that $\Mod(\Sigma_0,\vB(m))$ and $\Mod(\Sigma_{0,m}^{\circ})$ are isomorphic.  Recall that the action of $\Mod(\Sigma_{0,m}^{\circ})$ on $H_1(\Sigma_{0,m}^{\circ},\mathbb Z)$ induces a homomorphism $\Psi_m:\Mod(\Sigma_{0,m}^{\circ})\to S_m$ on the short exact sequence (\ref{permuting_branch_points}).  We will also  consider the map $\widehat \Psi_m:\Mod(\Sigma_0,\vB(m))\to S_m$ by precomposing $\Psi_m$ with the isomorphism  $\Mod(\Sigma_0,\vB(m))\cong\Mod(\Sigma_{0,m}^{\circ})$.  By Lemma \ref{homeo_lift} and the short exact sequence (\ref{permuting_branch_points}), $\PMod(\Sigma_0,\vB(m)) = \ker \widehat \Psi_m$ is contained in $\LMod_p(\Sigma_0,\vB(m))$.  This gives us the short exact sequence 
\[
1 \to \PMod(\Sigma_0,\vB(m)) \to \LMod_p(\Sigma_0,\vB(m)) \to \widehat \Psi_m(\LMod_p(\Sigma_0,\vB(m))) \to 1.
\]
Since $\widehat \Psi_m(\Mod(\Sigma_0,\vB(m))) \cong S_m$, the group  $\widehat \Psi_m(\LMod_p(\Sigma_0,\vB(m)))$ is isomorphic to a subgroup of $S_m$. Our next goal is to find the subgroup of $S_m$ isomorphic to $\widehat \Psi_m(\LMod_p(\Sigma_0,\vB(m)))$ where $p$ is a balanced superelliptic covering map.  

Let $p_{g,k}:\Sigma_g\to \Sigma_0$ be the balanced superelliptic cover.  We will denote \linebreak $\LMod_p(\Sigma_0,\vB)$ by $\LMod_{g,k}(\Sigma_0,\vB)$.  Recall that $\vB=\vB(2n+2)$ where $n = g/(k-1)$.  We will suppress the $2n+2$ in our notation, since the number of branch points of the balanced superelliptic covers is determined by $g$ and $k$.

\p{Parity of a permutation}
Fix an integer $m\geq 2$.  Let $\tau$ be a permutation in $S_m$.  We say that $\tau$ {\it preserves parity} if $\tau(q)=q\mod 2$ for all $q\in\{1,\cdots,m\}$.  We say that $\tau$ {\it reverses parity} if $\tau(q)\neq q\mod 2$ for all $q\in\{1,\cdots,m\}$.

Let $S_{2l}$ be the symmetric group on the set $\{1,\ldots,2l\}$.  Let $W_{2l}<S_{2l}$ be the subgroup consisting of permutations that either preserve parity, or reverse parity.  Then
\[
W_{2l} \cong (S_l\times S_l) \rtimes \ZZ/2\ZZ
\]
where $\ZZ/2\ZZ$ acts on $S_l \times S_l$ by switching the coordinates.

\begin{lem} \label{image_iso}
Let $\ol p_{g,k}:S_g\to S_0$ be a balanced superelliptic covering map of degree $k$.  Let $\Psi_{2n+2}:\Mod(\Sigma_0^{\circ})\to S_{2n+2}$ be the homomorphism induced from the action of $\Mod(\Sigma_{0}^{\circ})$ on $H_1(\Sigma_0^{\circ},\mathbb Z)$. If $k=2$, the image $\widehat \Psi_{2n+2}(\LMod_{g,2}(\Sigma_0,\vB)) = S_{2n+2}$.  For $k > 2$, $\widehat{\Psi}_{2n+2}(\LMod_{g,k}(\Sigma_0,\vB)) = W_{2n+2}$.
\end{lem}
Lemmas \ref{curve_lifting_criterion} and \ref{curve_lift} characterize the curves in $\Sigma_0^{\circ}$ that lift.  Let $\gamma$ be a curve in $\Sigma_0^{\circ}$ and let $T_{\gamma}$ be a Dehn twist about $\gamma$.  It is possible for $T_{\gamma}$ to lift to a homeomorphism of $\Sigma_g$ even if $\gamma$ does not lift.

\begin{proof}
Let $\gamma$ be a curve in $\Sigma_0^\circ$ and let $[\gamma] = \sum_{i=1}^{2n+1}\gamma_ix_i \in H_1(\Sigma_0^\circ;\ZZ)$.  Let $[f] \in \Mod(\Sigma_0^\circ)$ and let $\sigma=\Psi([f])\in S_{2n+2}$.  

Let $\gamma_{2n+2}=0$.  Then \[ [f(\gamma)]=\sum_{i=1}^{2n+1}(\gamma_{\sigma^{-1}(i)}-\gamma_j)x_i\] in homology, where $\sigma(j)=2n+2$.  Indeed, if $j=2n+2$, then
\[
[f(\gamma)] = \sum_{i=1}^{2n+1} \gamma_ix_{\sigma(i)} = \sum_{i=1}^{2n+1} \gamma_{\sigma^{-1}(i)}x_i = \sum_{i=1}^{2n+1}(\gamma_{\sigma^{-1}(i)}-\gamma_{2n+2})x_i.
\]
If $j \neq 2n+2$, let $\delta$ is a curve homotopic to the $(2n+2)$nd puncture.  Then $[\delta] = -\sum_{i=1}^{2n+1}x_i$.  Therefore $[f(x_j)] = -\sum_{i=1}^{2n+2}x_i$ and
\begin{align*}
[f(\gamma)] &= \sum_{i=1}^{j-1} \gamma_ix_{\sigma(i)} + \sum_{j+1}^{2n+1} \gamma_ix_{\sigma(i)} - \gamma_j\left(\sum_{i=1}^{2n+1} x_i\right) \\
&= \left(\sum_{\substack{i \in \{1, \ldots, 2n+1\} \\ i \neq \sigma(2n+2)}} (\gamma_{\sigma^{-1}(i)} - \gamma_j)x_i\right) - \gamma_jx_{\sigma(2n+2)}\\
&=\sum_{i=1}^{2n+1}(\gamma_{\sigma^{-1}(i)}-\gamma_j)x_i.
\end{align*}
The curve $\gamma$ lifts if and only if $\sum_{i=1}^{2n+2} (-1)^{i+1} \gamma_i \equiv 0 \mod k$.  
Let $f$ be a homeomorphism of $\Sigma_0^\circ$ and $\Psi([f]) = \sigma$ with $\sigma(j) = 2n+2$.  The image $f(\gamma)$ lifts if and only if $\sum_{i=1}^{2n+2}(-1)^{i+1}(\gamma_{\sigma^{-1}(i)} - \gamma_j) \equiv 0 \mod k$ by Lemma \ref{curve_lift}.

{\emph Case 1: $k=2$} \\
Let $[f] \in \Mod(\Sigma_0^{\circ})$ with $\Psi([f]) = \sigma$ such that $\sigma(j) = 2n+2$.  Observe that in $\ZZ/2\ZZ$, 
\[
\sum_{i=1}^{2n+2}(-1)^{i+1}(\gamma_{\sigma^{-1}(i)} - \gamma_j) = (2n+2)\gamma_j + \sum_{i=1}^{2n+2}(-1)^{i+1}\gamma_{\sigma^{-1}(i)} = \sum_{i=1}^{2n+2}(-1)^{i+1}\gamma_i
\]
so $\gamma$ lifts if and only if $f(\gamma)$ lifts.  We can then conclude $f$ lifts.  Therefore the image of $\left[f\right]$ under the isomorphism $\Mod(\Sigma_0^{\circ})\to\Mod(\Sigma_0,\vB)$ is in $\LMod(\Sigma_0,\vB)$. 

{\emph Case 2: $k \geq 3$}\\
Let $[f] \in \Mod(\Sigma_0^\circ)$ with $\Psi([f]) = \sigma$ such that $\sigma(j) = 2n+2$ and $\sigma \in W_{2n+2}$.  If $\sigma$ is parity preserving then 
\[
\sum_{i=1}^{2n+2}(-1)^{i+1}(\gamma_{\sigma^{-1}(i)} - \gamma_j) = \sum_{i=1}^{2n+2}(-1)^{i+1}\gamma_{\sigma^{-1}(i)} = \sum_{i=1}^{2n+2}(-1)^{i+1}\gamma_i
\]
so $\gamma$ lifts if and only if $f(\gamma)$ lifts.  If $\sigma$ is parity reversing, 
\[\sum_{i=1}^{2n+2}(-1)^{i+1}(\gamma_{\sigma^{-1}(i)} - \gamma_j) = \sum_{i=1}^{2n+2}(-1)^{i+1}\gamma_{\sigma^{-1}(i)} = -\sum_{i=1}^{2n+2}(-1)^{i+1}\gamma_i\]
so $\gamma$ lifts if and only if $f(\gamma)$ lifts.  If $\sigma$ is either parity reversing or parity preserving, then $f$ lifts.  Therefore the image of $\left[f\right]$ under the isomorphism $\Mod(\Sigma_0^{\circ})\to\Mod(\Sigma_0,\vB)$ is in $\LMod(\Sigma_0,\vB)$.

Conversely, assume that $\sigma\not\in W_{2n+2}$.  Then there exist odd integers $p$ and $q$ such that $\sigma(p)$ is odd and $\sigma(q)$ is even.  Without loss of generality, we may assume that $\sigma(p)=1,\sigma(q)=2$.  

Let $f\in\Mod(\Sigma_0,\vB)$ such that $\widehat\Psi_{2n+2}(f)=\sigma^{-1}$.  
We need to show that there exists some curve $\eta$ that lifts such that $f(\eta)$ does not lift.  Indeed, let $\eta=x_1+x_2$, then $\hat i(\eta,\alpha)=0$.  The homology class of $f(\eta)$ is $x_p+x_q$, but since both $p$ and $q$ are odd, $\hat i(f(\eta),\alpha)=2$.  Therefore $f(\eta)$ does not lift and the homeomorphism $f$ does not lift by Lemma \ref{homeo_lift}.
\end{proof}

Let $\ol p_{g,k}:S_g\to S_0$ be a balanced superelliptic covering map of degree $k$.
The short exact sequence (\ref{permuting_branch_points}) restricts to a short exact sequence:
\begin{equation}\label{lmod_sequence}
1\to \PMod(\Sigma_0,\vB)\to\LMod_{g,k}(\Sigma_0,\vB)\to W_{2n+2}\to 1\text{ 
for }k\geq 3.\end{equation}

Lemma \ref{image_iso} gives us the following result.  The case $k = 2$ has already been proven by Birman and Hilden \cite{BH2} using different methods.
\begin{sch} \label{index}
For $k = 2$, $\LMod_{g,k}(\Sigma_0,\vB) = \Mod(\Sigma_0,\vB)$.  For $k \geq 3$, the index $\left[\Mod(\Sigma_0,\vB):\LMod_{g,k}(\Sigma_0,\vB)\right]$ is $\frac{(2n+2)!}{2((n+1)!)^2}$.
\end{sch}
\begin{proof}
If $k=2$, we are in case 1 in the proof of \ref{image_iso}.  

For $k \geq 3$,
\begin{align*}
&\left[\Mod(\Sigma_0,\vB):\LMod_{g,k}(\Sigma_0,\vB)\right] \\
& \quad = \left[\Mod(\Sigma_0,\vB)/\PMod(\Sigma_0,\vB):\LMod_{g,k}(\Sigma_0,\vB)/\PMod(\Sigma_0,\vB)\right] \\ & \quad= \left[S_{2n+2}:W_{2n+2}\right].
\end{align*}
Observing that $\abs{W_{2n+2}} = 2((n+1)!)^2$ completes the proof.
\end{proof}

\section{Presentations of $\PMod(\Sigma_0,\vB(m))$ and $W_{2n+2}$}\label{PMod_and_W}
As in section \ref{lmod_sequence}, $\LMod_{g,k}(\Sigma_0,\vB(2n+2))$ can be written as a group extension of $W_{2n+2}$ by the pure mapping class group $\PMod(\Sigma_0,\vB(2n+2))$.  A presentation of $\PMod(\Sigma_0,\vB(2n+2))$ is found in Lemma \ref{purebraid}.   A presentation of $W_{2n+2}$ is found in Lemma \ref{Wpresentation}.  

\subsection{A presentation of $\PMod(\Sigma_0,\vB(m))$}\label{purebraid}
Let $D_m$ be a disk with $m$ marked points.  Number the marked points from $1$ to $m$.
Let $\sigma_i$ be the half-twist that exchanges the $i$th and $(i+1)$st marked points where the arc about which $\sigma_i$ is a half-twist in $\Sigma_0$ is shown in Figure \ref{all_gens}.  The pure braid group, denoted $PB_m$ is generated by elements $A_{i,j}$ with $1\leq i<j\leq m$ of the form:
$$A_{i,j}=(\sigma_{j-1}\cdots\sigma_{i+1})\sigma_i^2(\sigma_{j-1}\cdots\sigma_{i+1})^{-1}.$$

An example of of the curve about which $A_{i,j}$ is a twist is in Figure \ref{all_gens}.

\begin{lem} \label{pmod_pres}
The group $\PMod(\Sigma_0,\vB(m))$ is generated by $A_{i,j}$ for $1\leq i<j\leq m-1$ and has relations:
\begin{enumerate}
\item $\left[A_{p,q},A_{r,s}\right]=1 $ where $p<q<r<s$
\item $\left[A_{p,s},A_{q,r}\right]=1$ where $p<q<r<s$
\item $A_{p,r}A_{q,r}A_{p,q}=A_{q,r}A_{p,q}A_{p,r}=A_{p,q}A_{p,r}A_{q,r}$ where $p<q<r$
\item $\left[A_{r,s}A_{p,r}A_{r,s}^{-1},A_{q,s}\right]=1$ where $p<q<r<s$
\item $(A_{1,2}A_{1,3}\cdots A_{1,m-1})\cdots(A_{m-3,n-2}A_{m-3,n-1})(A_{m-2,m-1})=1$
\end{enumerate}
\end{lem}

\begin{proof}
Let $PB_m$ be the braid group on $m$ strands, which is isomorphic to the mapping class group of a disk $D_m$ with $m$ marked points.

By the capping homomorphism $Cap:PB_{m-1}\longrightarrow\PMod(\Sigma_0,\vB(m))$, there is a short exact sequence:
\begin{equation}\label{capsequence}1\longrightarrow \ZZ \longrightarrow PB_{m-1}\overset{Cap}\longrightarrow\PMod(\Sigma_0, \vB(m))\longrightarrow 1.\end{equation}
Here $\mathbb{Z}$ is generated by the Dehn twist about a curve homotopic to the boundary of $D_{m-1}$, which we will denote $T_{\beta}$.  From \cite[page 250]{primer} we have
\[
T_\beta = (A_{1,2}A_{1,3}\cdots A_{1,m})\cdots(A_{m-3,m-2}A_{m-3,m-1})(A_{m-2,m-1}).
\]
Using the presentation for $PB_m$ in Margalit--McCammond \cite[Theorem 2.3]{MM} and Lemma \ref{pres_quotient}, we obtain the desired presentation.
\end{proof}

\subsection{A presentation of $W_{2n+2}$}\label{Wpresentation}
As in section \ref{exact_LMod}, $W_{2n+2}$ is the subgroup the symmetric group $S_{2n+2}$ given by all permutations of $\{1,\ldots,2n+2\}$ that either preserve or reverse parity.

The symmetric group $S_m$ admits the presentation:

\begin{equation}\label{Sn_presentation}
S_{m} = \left \langle \tau_1,\ldots,\tau_{m-1} \mid \begin{cases} 
\tau_i^2 = 1 &\text{ for all $i \in \{1,\ldots,m-1\}$} \\
\tau_i\tau_{i+1}\tau_i=\tau_{i+1}\tau_i\tau_{i+1} & \text{ for all $i \in \{1,\ldots, m-2\}$} \\
[\tau_i,\tau_j] = 1 & \text{ for $\abs{i-j} > 1$}
\end{cases}\right\rangle
\end{equation}
where $\tau_i$ is the transposition $(i\,\,\, i+1)$.

\begin{lem} \label{W_pres}
Let $S_{2n+2}$ be the symmetric group on $\{1,\ldots,2n+2\}$.  Let $x_i = (2i-1 \,\,\, 2i+1)$, $y_i = (2i\,\,\,2i+2)$, and $z = (1\,\,\,2)\cdots(2n+1\,\,\,2n+2)$.  Then $W_{2n+2}$ admits a presentation with generators $\{x_1,\ldots,x_n,y_1,\ldots,y_n,z\}$ and relations
\begin{enumerate}
\item $[x_i,y_j] = 1$ for all $i,j \in \{1,\ldots, n\},$
\item $x_i^2 = 1$ and $y_i^2 = 1$ for all $i \in \{1,\ldots, n\},$ 
\item $x_ix_{i+1}x_i=x_{i+1}x_ix_{i+1}$ and $y_iy_{i+1}y_i=y_{i+1}y_iy_{i+1}$ for all $i \in \{1,\ldots, n-1\},$ 
\item $[x_i,x_j] = 1$ and $[y_i,y_j] = 1$ for all $\abs{i-j}\geq2,$
\item $z^2 = 1$, and
\item $zx_iz^{-1} = y_i$ for all $i \in \{1,\ldots, n\}$.
\end{enumerate}

\end{lem}
\begin{proof}
We have the short exact sequence
\[
1 \lra S_{n+1}\times S_{n+1} \overset{\alpha}{\lra} W_{2n+2} \overset{\pi}{\lra} \ZZ/2\ZZ \lra 1.
\]
The homomorphism $\alpha$ maps the first coordinate in $S_{n+1}\times S_{n+1}$ to permutations of $\{1,3,\ldots,2n+1\}$ and the second coordinate to permutations of $\{2,4,\ldots,2n+2\}$.  The map $\pi$ is given by $\pi(\sigma) = 0$ if $\sigma$ is parity preserving, and $\pi(\sigma) = 1$ if it is parity reversing.

Using the presentation (\ref{Sn_presentation}) for $S_{n+1}$ and Lemma \ref{pres_short_exact}, we find the presentation.  The details are left as an exercise.
\end{proof}

\section{Presentation of $\LMod_{g,k}(\Sigma_0,\vB)$}\label{mainproof}
In this section we compute a presentation for $\LMod_{g,k}(\Sigma_0,\vB)$, which is given in Theorem \ref{maintheorem}.  Throughout this section, let $p_{g,k}:\Sigma_g\to\Sigma_0$ be the balanced superelliptic cover of degree $k$.  Let $\vB$ be the set of $2n+2$ branch points in $\Sigma_0$.  

We apply Lemma \ref{pres_short_exact} to the short exact sequence (\ref{lmod_sequence}) from section \ref{exact_LMod}:
\[
1 \lra \PMod(\Sigma_0,\vB)\overset{\iota}{\longrightarrow} \LMod(\Sigma_0,\vB)\overset{\widehat{\Psi}_{2n+2}}{\lra} W_{2n+2}\lra 1.
\]
The inclusion map $\iota:\PMod(\Sigma_0,\vB)\to \LMod_{g,k}(\Sigma_0,\vB)$ maps the generators of $\PMod(\Sigma_0,\vB)$ to generators of $\LMod_{g,k}(\Sigma_0,\vB)$ by the identity. Thus the generators of $\PMod(\Sigma_0,\vB(2n+2))$ comprise the set $\wt S_K$ from Lemma \ref{pres_short_exact} in $\LMod_{g,k}(\Sigma_0,\vB)$.  Similarly, the relations of $\PMod(\Sigma_0,\vB)$ comprise the set $\wt R_K$ in $\LMod_{g,k}(\Sigma_0,\vB)$.

The generators $\wt S_H$ are the lifts in $\LMod_{g,k}(\Sigma_0,\vB)$ of the generators of $W_{2n+2}$.  The relations $R_1$ are the lifts in $\LMod_{g,k}(\Sigma_0,\vB)$ of the relations of $W_{2n+2}$.  We calculate the lifts in $\LMod_{g,k}(\Sigma_0,\vB)$ of both generators of $W_{2n+2}$ and relations of $W_{2n+2}$ in \ref{lifts}.

Finally the set $R_2$ is comprised of relations that come from conjugations of elements of $\wt S_K$ by elements in $\wt S_H$.  We calculate the conjugation relations in three steps in section \ref{conjugation}. 

\subsection{Lifts of generators and relations}\label{lifts} Let $\sigma_i$ be the half-twist that exchanges the $i$th and $i+1$st branch points about the arc in $\Sigma_0$ as in image Figure \ref{all_gens}.  The mapping class group $\Mod(\Sigma_0,\vB)$ admits the presentation \cite[page 122]{primer}
\[
\left\langle \sigma_1,\ldots,\sigma_{2n+1} \mid \begin{cases}
[\sigma_i,\sigma_j] = 1 & \abs{i - j} > 1, \\
\sigma_i\sigma_{i+1}\sigma_i = \sigma_{i+1}\sigma_i\sigma_{i+1} & i \in \{1,\ldots,2n\}, \\
(\sigma_1\sigma_2\cdots \sigma_{2n+1})^{2n+2} = 1, & \\
(\sigma_1 \cdots \sigma_{2n+1}\sigma_{2n+1}\cdots\sigma_1) = 1
\end{cases}\right\rangle.
\]

Since $\LMod_{g,k}(\Sigma_0,\vB)$ is a subgroup of $\Mod(\Sigma_0,\vB)$, we define the generators of $\LMod_{g,k}(\Sigma_0,\vB)$  in terms of the $\{\sigma_i\}$ in Lemma \ref{lmod_gens}.

\begin{lem}\label{lmod_gens}
The group $\LMod_{g,k}(\Sigma_0,\vB)$ is generated by 
\begin{enumerate}
\item $\{(\sigma_{j-1}\sigma_{j-2}\cdots \sigma_{i+1})\sigma_i^2(\sigma_{j-1}\sigma_{j-2}\cdots \sigma_{i+1})^{-1} : 1 \leq i < j \leq 2n+1\}$,
\item $\sigma_1\sigma_3\cdots \sigma_{2n+1}$,
\item $\{\sigma_{2i}\sigma_{2i-1}\sigma_{2i}^{-1} : i \in \{1,\ldots, n\}\}$, and
\item $\{\sigma_{2i+1}\sigma_{2i}\sigma_{2i+1}^{-1} :i \in \{1,\ldots, n\}\}$.
\end{enumerate}
\end{lem}
\begin{proof}
The elements from (1) are exactly the images of the generators for \linebreak $\PMod(\Sigma_0,\vB)$ from Lemma \ref{pmod_pres} under the inclusion map $\iota$.  The elements from (2) maps to $z$, the elements from (3) map to $x_i$ and the elements from (4) map to $y_i$ in Lemma \ref{W_pres}.  

Lemma \ref{pres_short_exact} tells us that the generators of types (1)-(4) suffice to form a generating set for $\LMod(\Sigma_0,\vB(2n+2))$.
\end{proof}

We will denote the generators by the following symbols.
\begin{equation}\label{gens}
\begin{split}
A_{i,j} &= (\sigma_{j-1}\sigma_{j-2}\cdots \sigma_{i+1})\sigma_i^2(\sigma_{j-1}\sigma_{j-2}\cdots \sigma_{i+1})^{-1}, \quad 1 \leq i < j \leq 2n+1 \\
c &= \sigma_1\sigma_3\cdots \sigma_{2n-1}\sigma_{2n+1} \\
a_\ell &= \sigma_{2\ell}\sigma_{2\ell-1}\sigma_{2\ell}^{-1}, \quad  \ell \in \{1,\ldots, n\}\\
b_\ell &= \sigma_{2\ell+1}\sigma_{2\ell}\sigma_{2\ell+1}^{-1}, \quad \ell \in \{1,\ldots, n\}.
\end{split}
\end{equation}

The generators $A_{i,j}$, $a_\ell$, and $b_\ell$ are all shown in Figure \ref{all_gens}.  The elements $a_\ell$ exchange consecutive odd marked points and the elements $b_\ell$ exchange consecutive even marked points.  The generator $c$ is the composition of half-twists about the arcs on the right side of Figure \ref{all_gens}, and $c$ switches each odd marked point with an even marked point.

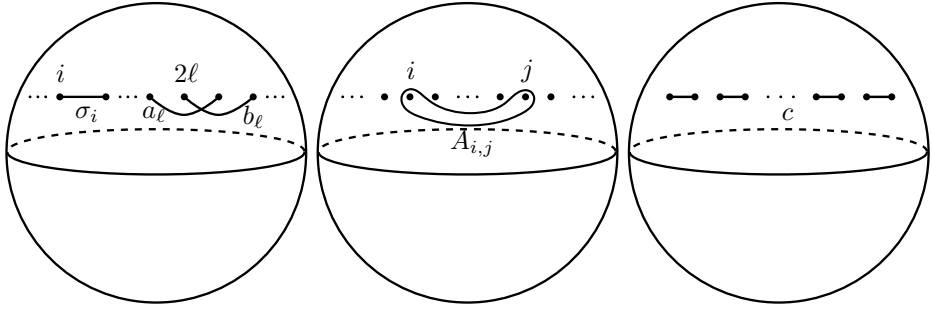
\begin{figure}[t]
\begin{center}
\labellist\small\hair 1.5pt
       \pinlabel {$\sigma_i$} at 28 65
       \pinlabel {$a_\ell$} at 51 65
       \pinlabel {$b_\ell$} at 85 64
       \pinlabel{$A_{i,j}$} at 160 55
       \pinlabel {$c$} at 268 65
       \pinlabel {$i$} at 19 79
       \pinlabel {$2\ell$} at 62 79
       \pinlabel {$i$} at 139 80
       \pinlabel {$j$} at 179 80
        \endlabellist
\includegraphics[scale=1.1]{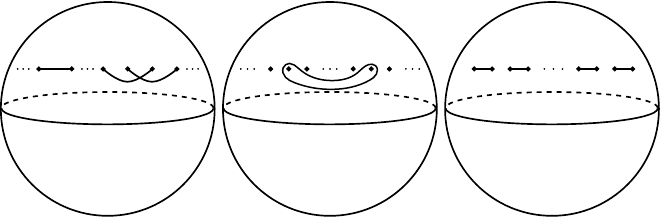}

\end{center}
    \caption{{\it Left to right}: the arcs about which $\sigma_i$, $a_\ell$, and $b_\ell$ are half-twists, the curve about which $A_{i,j}$ is a Dehn twist, and the collection of arcs about which $c$ is a composition of half-twists.  The labels above the marked points are to indicate the enumeration.}
    \label{all_gens}
\end{figure}

Although the elements $A_{i,2n+2} = (\sigma_{2n+1}\cdots \sigma_{i+1})\sigma_i^2(\sigma_{2n+1}\cdots \sigma_{i+1})^{-1}$ are in \linebreak $\PMod(\Sigma_0,\vB)$, they are not part of the generating set (\ref{gens}).  However, it will be useful to use the elements $A_{i,2n+2}$ in the set of relations for our final presentation.  The next lemma rewrites the elements $A_{i,2n+2}$ as words in the generators $A_{i,j}$ with $1\leq i<j\leq 2n+1$.

\begin{lem}\label{Abar}
Fix $\ell \in \{1,\ldots, 2n+1\}$.  Define
\[
\ol A_{i,j} := \begin{cases}
A_{i,j} & \text{if } j < \ell \\
A_{\ell,j+1}^{-1}A_{i,j+1}A_{\ell,j+1} &\text{if } i < \ell \leq j \\
A_{i+1,j+1} &\text{if } \ell \leq i.
\end{cases}
\] 
Then $A_{\ell,2n+2} = ( \ol A_{1,2} \cdots \ol A_{1,2n})(\ol A_{2,3} \cdots \ol A_{2,2n}) \cdots (\ol A_{2n-1,2n}).$
\end{lem}

To prove Lemma \ref{Abar}, we use the following facts:

Let $T_{r,s}=\sigma_{r+s}\sigma_{r+s-1}\cdots\sigma_{r}$, then:
\begin{enumerate}
\item $T_{r,2n-r}^{-1} A_{i,j} T_{r,2n-r} = A_{i,j}$ if $j < r$.
\item $T_{r,2n-r}^{-1} A_{i,j} T_{r,2n-r} = A_{r,j+1}^{-1}A_{i,j+1}A_{r,j+1}$ if $i < r \leq j$.
\item $T_{r,2n-r}^{-1} A_{i,j} T_{r,2n-r} = A_{i+1,j+1}$ if $r \leq i$.
\end{enumerate}

These facts can be checked through routine, but lengthy computation.

\begin{proof}
We first show that $$A_{2n+1,2n+2} = (A_{1,2} \cdots A_{1,2n})(A_{2,3} \cdots A_{2,2n}) \cdots (A_{2n-1,2n}).$$  Indeed, let $\gamma$ be the curve about which $A_{2n+1,2n+2}$ is a twist.  The curve $\gamma$ is a separating curve in the sphere with $2n+2$ marked points.  On one side, $\gamma$ bounds a disk containing the first $2n$ marked points and on the other side, $\gamma$ bounds a disk containing the $n+1$st and $n+2$nd marked points.  A Dehn twist about the boundary of the disk with $2n$ marked points can be written as $$(A_{1,2} \cdots A_{1,2n})(A_{2,3} \cdots A_{2,2n}) \cdots (A_{2n-2,2n-1}A_{2n-2,2n})(A_{2n-1,2n}),$$ as seen in Farb and Margalit \cite[page 260]{primer}.

Then let $T = T_{\ell,2n+2}$.  Notice that $A_{\ell,2n+2} = T^{-1}A_{2n+1,2n+2}T$.  Facts (1), (2), and (3) above show $T^{-1}A_{i,j}T = \ol A_{i,j}$ for all $1 \leq i < j \leq 2n$.  Therefore
\begin{align*}
A_{\ell,2n+2} &= T^{-1}A_{2n+1,2n+2}T \\
&= T^{-1}A_{1,2}TT^{-1} \cdots TT^{-1}A_{1,2n}TT^{-1}\cdots\\ 
&\quad \quad \quad TT^{-1}A_{2n-2,2n-1}TT^{-1}A_{2n-2,2n}TT^{-1}A_{2n-1,2n}T \\
&= ( \ol A_{1,2} \cdots \ol A_{1,2n})(\ol A_{2,3} \cdots \ol A_{2,2n}) \cdots (\ol A_{2n-2,2n-1}\ol A_{2n-2,2n})(\ol A_{2n-1,2n}).
\end{align*}
\end{proof}

The relations of $R_1$ of Lemma \ref{pres_short_exact} are given in Lemma \ref{R_1relations}.  To consolidate the family of commutator relations, let \begin{equation}\label{commutators}
C_{i,j}=\begin{cases}
A_{2i-1,2i}^{-1}A_{2i+1,2i+2}^{-1}A_{2i-1,2i+2}A_{2i,2i+1}&\text{ if }i=j\\
A_{2i+1,2i+2}^{-1}A_{2i,2i+3}^{-1}A_{2i+2,2i+3}A_{2i,2i+1}&\text{ if }i=j+1\\
1&\text{ otherwise}
\end{cases}
\end{equation}
for $1\leq i\leq j\leq n$.

\begin{lem}\label{R_1relations} Let $A_{i,j},a_\ell,b_\ell, $ and $c$ be the generators defined in (\ref{gens}).
The following relations hold.

\hspace{22pt}

{\bf Commutator relations}
\begin{enumerate}
\item $[a_i,b_j] = C_{i,j}$ where $C_{i,j}$ is given by (\ref{commutators})\\
{\bf Braid relations}
\item $a_ia_{i+1}a_i = a_{i+1}a_ia_{i+1}$ and $b_ib_{i+1}b_i = b_{i+1}b_ib_{i+1}$ for $i \in \{1,\ldots, n-1\}$
\item $[a_i,a_j] = [b_i,b_j] = 1$ if $\abs{j-i}>1$\\
{\bf Half twists squared are Dehn twists}
\item $a_i^2 = A_{2i-1,2i+1}$ and $b_i^2 = A_{2i,2i+2}$ for $i \in \{1,\ldots, n\}$.
\item $c^2 = A_{1,2}A_{3,4}\cdots A_{2n+1,2n+2}$\\
{\bf Parity Flip}
\item $ca_ic^{-1}b_i^{-1} = 1$.
\end{enumerate}

\end{lem}


\begin{proof}
Since the $\sigma_i$ satisfy the braid relations, we will repeatedly use the following modification of braid relations, which we will put in parentheses:
\begin{itemize} 
\item $\sigma_{m}\sigma_{m+1}\sigma_m^{-1}=\sigma_{m+1}^{-1}\sigma_m\sigma_{m+1}$, \item$\sigma_{m}\sigma_{m+1}^{-1}\sigma_m^{-1}=\sigma_{m+1}^{-1}\sigma_m^{-1}\sigma_{m+1}$ and 
\item $\sigma_{m}^{-1}\sigma_{m+1}^2\sigma_{m}=\sigma_{m+1}\sigma_{m}^2\sigma_{m+1}^{-1}$
\end{itemize}

We will also use Dehornoy's handle reduction \cite{dehorn} and will underline the handles.

For (1) It suffices to show $[a_i,b_j]C_{i,j}^{-1}$ is the identity.  

If $i = j$
\begin{align*}
&a_ib_ia_i^{-1}b_i^{-1}A_{2i,2i+1}^{-1}A_{2i-1,2i+2}^{-1}A_{2i+1,2i+2}A_{2i-1,2i} \\
&= (\sigma_{2i}\sigma_{2i-1}\underline{\sigma_{2i}^{-1})(\sigma_{2i + 1}\sigma_{2i}}\sigma_{2i+1}^{-1})(\sigma_{2i}\sigma_{2i-1}^{-1}\sigma_{2i}^{-1})(\sigma_{2i+1}\sigma_{2i}^{-1}\sigma_{2i+1}^{-1}) \\ & \quad\quad (\sigma_{2i}^{-1}\underline{\sigma_{2i}^{-1})(\sigma_{2i+1}\sigma_{2i}}\sigma_{2i-1}^{-1}\underline{\sigma_{2i-1}^{-1}\sigma_{2i}^{-1}\sigma_{2i+1}^{-1})(\sigma_{2i+1}^2)(\sigma_{2i-1}}\sigma_{2i-1}) \\
&= \sigma_{2i} \underline{\sigma_{2i-1} \sigma_{2i+1} \sigma_{2i} \sigma_{2i+1}^{-1} \sigma_{2i+1}^{-1} \sigma_{2i} \sigma_{2i-1}^{-1}} \sigma_{2i}^{-1} \sigma_{2i+1} \sigma_{2i}^{-1} \sigma_{2i+1}^{-1} \\ & \quad \quad \underline{\sigma_{2i}^{-1} \sigma_{2i+1} \sigma_{2i}} \sigma_{2i+1}^{-1} \sigma_{2i-1}^{-1} \sigma_{2i} \underline{\sigma_{2i-1}^{-1} \sigma_{2i}^{-1} \sigma_{2i+1} \sigma_{2i-1}} \\
&= \underline{\sigma_{2i} \sigma_{2i+1} \sigma_{2i}^{-1}} \sigma_{2i-1} \underline{\sigma_{2i} \sigma_{2i+1}^{-1} \sigma_{2i+1}^{-1} \sigma_{2i}^{-1}} \,\, \underline{\sigma_{2i-1} \sigma_{2i+1}^{-1} \sigma_{2i-1}^{-1}} \sigma_{2i} \sigma_{2i} \sigma_{2i-1}^{-1} \sigma_{2i}^{-1} \sigma_{2i+1} \\
&= \sigma_{2i+1}^{-1} \sigma_{2i} \sigma_{2i+1} \underline{\sigma_{2i-1} \sigma_{2i+1}^{-1} \sigma_{2i-1}^{-1}} \sigma_{2i}^{-1} \sigma_{2i+1} \\
&= 1
\end{align*}
A similar computation estabilshes the relation if $i = j+1$.  

For the cases where $i\notin\{j,j+1\}$ we note $\abs{2j-2i} \geq 2$ and $\abs{(2i-1) - (2j+1)} \geq 2$.  Therefore $a_i = \sigma_{2i}\sigma_{2i-1}\sigma_{2i}^{-1}$ commutes with $b_j = \sigma_{2j+1}\sigma_{2j}\sigma_{2j+1}^{-1}$, establishing relation (1). 

For relation (2) we have
\begin{align*}
&a_ia_{i+1}a_ia_{i+1}^{-1}a_i^{-1}a_{i+1}^{-1} \\
&= (\sigma_{2i} \sigma_{2i-1} \sigma_{2i}^{-1})(\sigma_{2i+2} \sigma_{2i+1} \sigma_{2i+2}^{-1})(\sigma_{2i} \sigma_{2i-1} \sigma_{2i}^{-1})\\ & \quad \quad \quad   (\sigma_{2i+2} \sigma_{2i+1}^{-1} \sigma_{2i+2}^{-1})(\sigma_{2i} \sigma_{2i-1}^{-1} \sigma_{2i}^{-1})(\sigma_{2i+2} \sigma_{2i+1}^{-1} \sigma_{2i+2}^{-1})\\
&=\sigma_{2i+2} (\sigma_{2i} \sigma_{2i-1} \sigma_{2i+1} \sigma_{2i}\sigma_{2i+1}^{-1}) \sigma_{2i-1} \sigma_{2i}^{-1}\sigma_{2i+2}^{-1}\\ & \quad \quad \quad   \sigma_{2i+2} \sigma_{2i+1}^{-1} (\sigma_{2i-1}^{-1} \sigma_{2i}^{-1} \sigma_{2i-1})(\sigma_{2i+2}^{-1}\sigma_{2i+2}) \sigma_{2i+1}^{-1} \sigma_{2i+2}^{-1}\\
&=\sigma_{2i+2} \sigma_{2i} \sigma_{2i+1} (\sigma_{2i} \sigma_{2i-1} \sigma_{2i}) (\sigma_{2i}^{-1}\sigma_{2i+1}^{-1}\sigma_{2i}^{-1})\sigma_{2i-1}^{-1} \sigma_{2i}^{-1} \sigma_{2i-1}\sigma_{2i+1}^{-1} \sigma_{2i+2}^{-1}\\
&=\sigma_{2i+2} (\sigma_{2i} \sigma_{2i+1} \sigma_{2i} )\sigma_{2i-1} \sigma_{2i+1}^{-1}\sigma_{2i}^{-1})\sigma_{2i-1}^{-1} \sigma_{2i}^{-1} \sigma_{2i-1}\sigma_{2i+1}^{-1} \sigma_{2i+2}^{-1}\\
&=\sigma_{2i+2} \sigma_{2i} \sigma_{2i+1} \sigma_{2i} \sigma_{2i-1} \sigma_{2i+1}^{-1}(\sigma_{2i-1}^{-1}\sigma_{2i}^{-1} \sigma_{2i-1}^{-1}) \sigma_{2i-1}\sigma_{2i+1}^{-1} \sigma_{2i+2}^{-1}\\
&=\sigma_{2i+2} \sigma_{2i} \sigma_{2i+1} (\sigma_{2i+1}^{-1} \sigma_{2i}^{-1}\sigma_{2i+1})\sigma_{2i+1}^{-1} \sigma_{2i+2}^{-1}\\
&=1\\
\end{align*}
The same proof can be used for the relation $b_ib_{i+1}b_i=b_{i+1}b_ib_{i+1}$ by decreasing all indices by 1. Relation (3) can also be proved using a similar argument.

For (4) we see $a_i^2 = \sigma_{2i}\sigma_{2i-1}^2 \sigma_{2i}^{-1} = A_{2i-1,2i+1}$. Similarly for (5), we have $b_i^2 = A_{2i,2i+2}$.  For (6) we have
\begin{align*}
ca_ic^{-1} &= (\sigma_1 \sigma_3 \cdots \sigma_{2i-1} \sigma_{2i+1} \cdots \sigma_{2n+1})(\sigma_{2i}\sigma_{2i-1}\sigma_{2i}^{-1}) \\ & \quad \quad \quad \quad \quad \quad \quad (\sigma_1^{-1} \sigma_3^{-1} \cdots \sigma_{2i-1}^{-1} \sigma_{2i+1}^{-1} \cdots \sigma_{2n+1}^{-1}) \\
&= \sigma_{2i+1}\sigma_{2i-1}\sigma_{2i}\sigma_{2i-1}\sigma_{2i}^{-1}\sigma_{2i-1}^{-1} \sigma_{2i+1}^{-1} \\
&= \sigma_{2i + 1} \sigma_{2i} \sigma_{2i+1}^{-1} \\
&= b_i
\end{align*}
which completes the proof.
\end{proof}


\p{Topological interpretation} Although the proof of Lemma \ref{R_1relations} is purely algebraic, there are topological interpretations of most of the relations.  Let $\gamma_\ell$ be the arc about which $a_\ell$ is a half-twist, and $\delta_\ell$ the arc about which $b_\ell$ is a half-twist.

When $i \neq j,j+1$, $\gamma_i$ and $\delta_j$ can be modified by homotopy to be disjoint so the relations $[a_i,b_j] = 1$ in (1) hold.  The homeomorphisms $\{a_i\}$ are supported on a closed neighborhood of the union $\gamma_1 \cup \cdots \cup \gamma_n$, which is an embedded disk $D_{n+1}$ with $n+1$ marked points.  The mapping class group of $D_{n+1}$ is isomorphic to the braid group $B_{n+1}$.  Embedding $D_{n+1}$ in $\Sigma_0$ with $2n+2$ marked points induces a homomorphism $\iota: B_{n+1} \to \Mod(\Sigma_0,\vB(2n+2))$.  The homomorphism $\iota$ maps the standard braid generators to the $a_i$, and so the braid relations (2) and (3) hold.  The same applies to the $b_i$.      

Relations (4) and (5) reflect the fact that squaring a half-twist about an arc is homotopic to a Dehn twist about a curve surrounding the arc.  Recall that if $\tau_\gamma$ is a half-twist about an arc $\gamma$ in $\Sigma_0$ and $f$ is a homeomorphism of $\Sigma_0$, then $f^{-1}\tau_\gamma f = \tau_{f(\gamma)}$.  We realize $ca_ic^{-1} = b_i$ in (6) by applying the homeomorphism $c^{-1}$ to the arc $\gamma_i$, where $\gamma_i$ is the arc about which $a_i$ is a half-twist.  

\subsection{Conjugation relations}\label{conjugation}
We now shift our attention to finding the relations that comprise $R_2$ from Lemma \ref{pres_short_exact}.  Lemmas \ref{conjugationc},  \ref{conjugationa}, and \ref{conjugationb} give us the conjugation relations.  

First we consider conjugation of the pure braid group generators by $c$.  Let \begin{equation}\label{Xij}
X_{i,j}=\begin{cases}
A_{i,j} &\text{for odd $i$, $j = i+1$}\\
A_{i+1,j+1} & \text{for odd $i, j$} \\
(A_{i-1,j}A_{i-1,i}^{-1})^{-1}A_{i-1,j-1} (A_{i-1,j}A_{i-1,i}^{-1})& \text{for even $i,j$} \\
A_{i,j+1}^{-1}A_{i-1,j+1}A_{i,j+1} &\text{for even $i$, odd $j$} \\
A_{j-1,j}A_{i+1,j-1}A_{j-1,j}^{-1} &\text{otherwise}.
\end{cases}\end{equation}

\begin{lem}\label{conjugationc}
For $1 \leq i < j \leq 2n+1$, let $A_{i,j}$ and $c$ be as above.  Then $$cA_{i,j}c^{-1} = X_{i,j}$$
where the $X_{i,j}$ are as in (\ref{Xij}).
\end{lem}

\begin{proof}
Recall that $A_{i,j} = (\sigma_{j-1}\sigma_{j-2}\cdots \sigma_{i+1})\sigma_i^2(\sigma_{j-1}\sigma_{j-2}\cdots \sigma_{i+1})^{-1}$ for $1 \leq i < j \leq 2n+1$ and 
$c = \sigma_1\sigma_3\cdots \sigma_{2n-1}\sigma_{2n+1}$.

 Let $\gamma_{i,j}$ be the simple closed curve in $\Sigma_0$ about which $A_{i,j}$ is a Dehn twist.  Let $T_{\gamma_{i,j}}=A_{i,j}$.  Recall that $cT_{\gamma_{i,j}}c^{-1}=T_{c^{-1}(\gamma_{i,j})}$ (where we maintain our convention that we read products from left to right).  Therefore to prove the lemma, it suffices to show that $c^{-1}(\gamma_{i,j})$ is the curve about which $X_{i,j}$ is a twist.  The homeomorphism $c^{-1}$ is the product (counterclockwise) twists $\sigma_1,\sigma_3,\cdots,\sigma_{2n+1}$.  The cases for the image of $c^{-1}(\gamma_{i,j})$ depend on the signs of the intersections of arcs about which $c$ is a twist and $\gamma_{i,j}$.
 
 We first note that when $i$ is odd and $j=i+1$ the curve $\gamma_{i,j}$ is disjoint from $c$, therefore $cA_{i,j} c^{-1}=A_{i,j}$.  We then consider the remaining cases.
 
 \p{$i$ and $j$ are both odd} The curves $c^{-1}(\gamma_{i,j})$  and $\gamma_{i+1,j+1}$ are shown to be isotopic in Figure \ref{odd}.  Therefore $cA_{i,j}c^{-1}=A_{i+1,j+1}$.
 
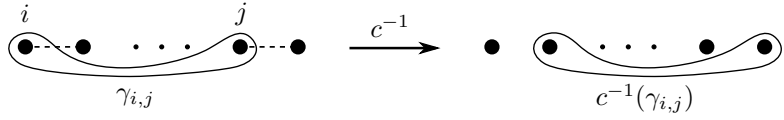
\begin{figure}[t]
\begin{center}
\labellist\small\hair 2.5pt
    \pinlabel {$i$} at 4 15
    \pinlabel {$j$} at 55 15
    \pinlabel {$c^{-1}$} at 90 12
    \pinlabel {$\gamma_{i,j}$} at 30 -5
    \pinlabel {$c^{-1}(\gamma_{i,j})$} at 150 -5
    \endlabellist
\includegraphics[scale=1.6]{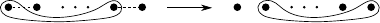}
  \end{center}
    \caption{{\it Both $i$ and $j$ are odd}.  The curve $\gamma_{i,j}$ and its image under $c^{-1}$, which is a product of half-twists about the dashed arcs.}
    \label{odd}
\end{figure}

  \begin{figure}[t]
\begin{center}
\labellist\small\hair 2.5pt
	\pinlabel {$i$} at 18 55
	\pinlabel {$j$} at 69 55
	\pinlabel {$\gamma_{i,j}$} at 35 53
	\pinlabel {$c^{-1}(\gamma_{i,j})$} at 65 22
	\pinlabel {$c^{-1}$} at 30 26
	\pinlabel {$\gamma_{i-1,j-1}$} at 130 53
	\pinlabel {$A_{i-1,i}^{-1}(A_{i-1,j}(\gamma_{i-1,j-1}))$} at 175 22
	\pinlabel {$A_{i-1,j}$} at 180 55
	\pinlabel {$A_{i-1,i}^{-1}$} at 235 26
        \endlabellist
\includegraphics[scale=1.4]{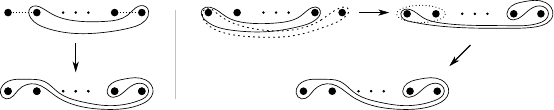}
  \end{center}
    \caption{{\it Both $i$ and $j$ are even}.  The left figure shows $\gamma_{i,j}$ and its image under $c^{-1}$, which is a product of half-twists about the dashed arcs.  Starting at the top left and going clockwise, the right figure shows $\gamma_{i-1,j-1}$, $A_{i-1,j}(\gamma_{i-1,j-1})$, and $A_{i-1,i}^{-1}(A_{i-1,j}(\gamma_{i-1,j-1}))$.  The dashed curves indicate the curves about which $A_{i-1,j}$ and $A_{i-1,i}^{-1}$ are Dehn twists.}
    \label{even}
\end{figure}
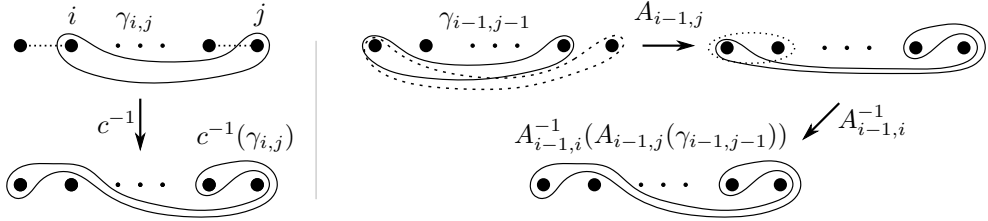

\p{$i$ and $j$ are both even} The curves $c^{-1}(\gamma_{i,j})$ and $A_{i-1,i}^{-1}(A_{i-1,j}(\gamma_{i-1,j-1}))$ (with composition applied as indicated) are shown to be isotopic in Figure \ref{even}.  Therefore $cA_{i,j}c^{-1} =   (A_{i-1,j}A_{i-1,i}^{-1})^{-1}A_{i-1,j-1} (A_{i-1,j}A_{i-1,i}^{-1})$.
 
 \p{$i$ is even and $j$ is odd} The curves $c^{-1}(\gamma_{i,j})$ and $A_{i,j+1}(\gamma_{i-1,j+1})$ are shown to be isotopic in Figure \ref{evenioddj}.  Therefore $cA_{i,j}c^{-1}= A_{i,j+1}^{-1}A_{i-1,j+1}A_{i,j+1}$.
  
  \begin{figure}[t]
 \begin{center}
\labellist\small\hair 1pt
    \pinlabel {$i$} at 17 50
    \pinlabel {$j$} at 55 50
    \pinlabel {$c^{-1}$} at 43 23
    \pinlabel {$\gamma_{i,j}$} at 25 33
    \pinlabel {$c^{-1}(\gamma_{i,j})$} at 12 15
    \pinlabel {$\gamma_{i-1,j+1}$} at 108 31
    \pinlabel {$A_{i,j+1}$} at 144 22
    \pinlabel {$A_{i,j+1}(\gamma_{i-1,j+1})$} at 108 15
    \endlabellist
\includegraphics[scale=1.6]{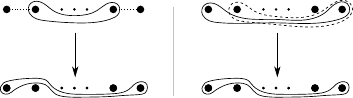}
  \end{center}
    \caption{{\it Even $i$ and odd $j$}. The left figure shows $\gamma_{i,j}$ and its image under $c^{-1}$, which is a product of half-twists about the dashed arcs.  The right figure shows $\gamma_{i-1,j+1}$ and its image under $A_{i,j+1}$, which is a Dehn twist about the dashed curve.}
    \label{evenioddj}
\end{figure}
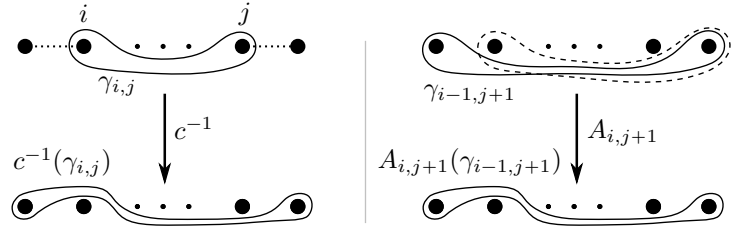

\begin{figure}[t]
\begin{center}
\labellist\small\hair 1pt
    \pinlabel {$i$} at 4 51
    \pinlabel {$j$} at 68 51
    \pinlabel {$c^{-1}$} at 43 23
    \pinlabel {$A_{j-1,j}^{-1}$}  at 144 23
    \pinlabel {$\gamma_{i,j}$} at 10 33
    \pinlabel {$c^{-1}(\gamma_{i,j})$} at 65 14
    \pinlabel {$\gamma_{i+1,j-1}$} at 120 33
    \pinlabel {$A_{j-1,j}^{-1}(\gamma_{i+1,j-1})$} at 108 14
    \endlabellist
\includegraphics[scale=1.6]{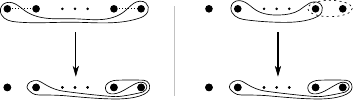}
  \end{center}
      \caption{{\it Odd $i$ and even $j$, $j \neq i+1$}. The left figure shows $\gamma_{i,j}$ and its image under $c^{-1}$, which is a product of half-twists about the dashed arcs.  The right figure shows $\gamma_{i+1,j-1}$ and its image under $A_{j-1,j}^{-1}$, which is an inverse Dehn twist about the dashed curve.}
    \label{oddievenj}
\end{figure}
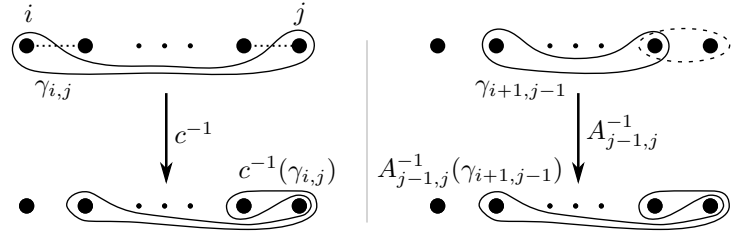

\p{$i$ is odd and $j$ is even $j\neq i+1$} The curves $c^{-1}(\gamma_{i,j})$ and $A_{j-1,j}^{-1}(\gamma_{i+1,j-1})$ are shown to be isotopic in Figure \ref{oddievenj}.  Therefore $cA_{i,j}c^{-1} = A_{j-1,j}A_{i+1,j-1}A_{j-1,j}^{-1}$.
\end{proof}


Next we consider conjugation of the elements $A_{i,j}$ by the generators $a_{\ell}$.  The resulting relations (along with the conjugates by $b_{\ell}$ in Lemma \ref{conjugationb}) correspond to the conjugates of the words in $\wt S_K$ by the words in $\wt S_H$.  Let \begin{equation}\label{Yij}
Y_{i,j,\ell}=\begin{cases}
A_{i,j}&\text{ if }i < 2\ell-1, j > 2\ell+1, \\
A_{i,j}&\text{ if }i,j > 2\ell+1\text{ or }i,j< 2\ell-1\\
A_{i,j+2}&\text{ if }i < 2\ell-1,j=2\ell-1\\
 (A_{i,j-1}^{-1}A_{i,j+1})^{-1}A_{i,j}(A_{i,j-1}^{-1}A_{i,j+1})&\text{ if }i < 2\ell-1, j=2\ell\\
 A_{i,j}^{-1}A_{i,j-2}A_{i,j}&\text{ if }i <2\ell-1, j=2\ell+1\\
A_{i,j+1}A_{j,j+1}A_{i,j+1}^{-1}&\text{ if }i = 2\ell-1,j = 2\ell\\
A_{i,j}&\text{ if }i = 2\ell-1, j = 2\ell+1\\
A_{i+2,j}&\text{ if }i = 2\ell-1,j>2\ell+1 \\
A_{i-1,j-1}&\text{ if }i = 2\ell, j = 2\ell+1\\
(A_{i,i+1}^{-1}A_{i-1,i})^{-1}A_{i,j}(A_{i,i+1}^{-1}A_{i-1,i})&\text{ if }i = 2\ell,j>2\ell+1\\
A_{i,j}^{-1}A_{i-2,j}A_{i,j}&\text{ if }i=2\ell+1, j>2\ell+1
\end{cases}\end{equation}
\begin{lem}\label{conjugationa}
For $1 \leq i < j \leq 2n+1$ and $\ell \in \{1,\ldots, n\}$, let $A_{i,j}$ and $a_\ell$ be as above.  Then $$a_\ell A_{i,j}a_\ell^{-1} = Y_{i,j,\ell}$$
where the $Y_{i,j,\ell}$ are as in (\ref{Yij}).
\end{lem}




\begin{proof}
Recall that $a_\ell=\sigma_{2\ell}\sigma_{2\ell-1}\sigma_{2\ell}^{-1}$ and 
\[
A_{i,j}=(\sigma_{j-1}\sigma_{j-2}\cdots \sigma_{i+1})\sigma_i^2(\sigma_{j-1}\sigma_{j-2}\cdots \sigma_{i+1})^{-1} .
\]
The transpositions $\sigma_p$ and $\sigma_q$ commute if $|p-q|\geq 2$.  Therefore $a_{\ell}$ and $A_{i,j}$ commute if both $(2\ell-1)-(j-1)\geq 2$ and $2\ell-(i+1)\geq 2$, if $i,j>2\ell+1$, or if $i,j<2\ell-1$. 

Therefore in the first two cases of \ref{Yij}, $a_{\ell}$ and $A_{i,j}$ commute.  In the remaining cases, at least one of $i$ and $j$ is equal to $2\ell-1,2\ell,$ or $2\ell+1$.

If $i=2\ell-1$ and $j=i+2=2\ell+1$, then $A_{i,j}=\sigma_{i+1}\sigma_i^{2}\sigma_{i+1}^{-1}=A_{i,i+2}$ and
\begin{align*}
a_{\ell}A_{i,j}a_\ell^{-1}&=(\sigma_{i+1}\sigma_{i}\sigma_{i+1}^{-1})\sigma_{i+1}\sigma_i^2\sigma_{i+1}^{-1}(\sigma_{i+1}\sigma_{i}^{-1}\sigma_{i+1}^{-1})\\
&=\sigma_{i+1}\sigma_i^2\sigma_{i+1}^{-1}=A_{i,i+2}\\
&=A_{i,j}
\end{align*}

If $i=2\ell$ and $j=i+1=2\ell+1$, then $a_{\ell}=\sigma_i\sigma_{i-1}\sigma_i^{-1}$ and $A_{i,j}=\sigma_i^2$.  Then:
\begin{align*}
a_{\ell}A_{i,j}a_\ell^{-1}&=(\sigma_{i}\sigma_{i-1}\sigma_{i}^{-1})\sigma_i^2(\sigma_{i}\sigma_{i-1}^{-1}\sigma_{i}^{-1})\\
&=(\sigma_{i-1}\sigma_{i}\sigma_{i-1})(\sigma_{i-1}^{-1}\sigma_{i}^{-1}\sigma_{i-1})\\
&=\sigma_{i-1}^2=A_{i-1,i}\\
&=A_{i-1,j-1}
\end{align*}

If $i=2\ell-1$ and $j=i+1=2\ell$, then $A_{i,j}=\sigma_i^2$ and $a_{\ell}=\sigma_{i+1}\sigma_i\sigma_{i+1}^{-1}$.
\begin{align*}
a_{\ell}A_{i,j}a_\ell^{-1}&=(\sigma_{i+1}\sigma_{i}\sigma_{i+1}^{-1})\sigma_i^2(\sigma_{i+1}\sigma_{i}^{-1}\sigma_{i+1}^{-1})\\
&=\sigma_{i+1}\sigma_{i}\sigma_{i+1}^{-1}\sigma_i(\sigma_{i+1}\sigma_{i+1}^{-1})\sigma_i\sigma_{i+1}\sigma_{i}^{-1}\sigma_{i+1}^{-1}\\
&=\sigma_{i+1}\sigma_{i}(\sigma_{i}\sigma_{i+1}\sigma_{i}^{-1})(\sigma_{i}\sigma_{i+1}\sigma_{i}^{-1})\sigma_{i}^{-1}\sigma_{i+1}^{-1}\\
&=(\sigma_{i+1}\sigma_{i}^2\sigma_{i+1}^{-1})\sigma_{i+1}^2(\sigma_{i+1}\sigma_{i}^{-2}\sigma_{i+1}^{-1})\\
&=A_{i,i+2}A_{i+1,i+2}A_{i,i+2}^{-1}\\
&=A_{i,j+1}A_{j,j+1}A_{i,j+1}^{-1}
\end{align*}

We prove the remaining cases topologically.  Let $\gamma_{i,j}$ be the curve about which $A_{i,j}$ is the Dehn twist.  Let $T_{\gamma_{i,j}}=A_{i,j}$.  For any mapping class $f\in \Mod(\Sigma_0,\vB(2n+2))$, the conjugation $fT_{\gamma_{i,j}}f^{-1}=T_{f^{-1}(\gamma_{i,j})}$ (where we maintain our convention that we read products from left to right).  Each of the conjugates $a_{\ell}A_{i,j}a_{\ell}^{-1}$ is a Dehn twist about the curve $a_{\ell}^{-1}(\gamma_{i,j})$.  To prove the lemma, it suffices to prove that the homeomorphism $Y_{i,j}$ is $T_{a_{\ell}^{-1}(\gamma_{i,j})}$ for various relationships between $i,j$ and $\ell$.

\p{The case where $i<2\ell-1$ and $j=2\ell-1$} The curves $a_\ell^{-1}(\gamma_{i,j})$ and $\gamma_{i,j+2}$ are shown to be isotopic in Figure \ref{j2l-1}.  Therefore $a_\ell A_{i,j}a_\ell^{-1} = A_{i,j+2}$.

  \begin{figure}[t]
  \begin{center}
\labellist\small\hair 1pt
    \pinlabel {$i$} at 4 15
    \pinlabel {$j$} at 41 15
    \pinlabel {$a_{\ell}^{-1}$} at 90 13
    \pinlabel {$\gamma_{i,j}$} at 22 12
    \pinlabel {$a_{\ell}^{-1}(\gamma_{i,j})$} at 135 13
    \endlabellist
\includegraphics[scale=1.6]{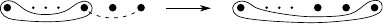}
  \end{center}
    \caption{$i< 2\ell -1, j = 2\ell - 1$.  The curve $\gamma_{i,j}$ and its image under $a_{\ell}^{-1}$, which is a counterclockwise half twist about the dashed arc.}
    \label{j2l-1}
\end{figure}
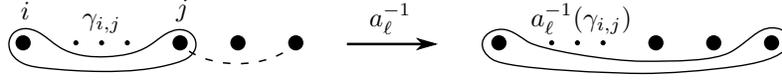

\p{The case where $i<2\ell-1$ and $j=2\ell$} The curves $a_{\ell}^{-1}(\gamma_{i,j})$ and $A_{i,j+1}(A_{i,j-1}^{-1}(\gamma_{i,j}))$ are shown to be isotopic in Figure \ref{j2l}.  Therefore
\[
a_{\ell}A_{i,j}a_{\ell}^{-1}=(A_{i,j-1}^{-1}A_{i,j+1})^{-1}A_{i,j}(A_{i,j-1}^{-1}A_{i,j+1}).
\]
\begin{figure}[t]
\begin{center}
\labellist\small\hair 1pt
    \pinlabel {$i$} at 4 54
    \pinlabel {$j$} at 55 54
    \pinlabel {$a_{\ell}^{-1}$} at 43 26
    \pinlabel {$\gamma_{i,j}$} at 10 33
    \pinlabel {$a_{\ell}^{-1}(\gamma_{i,j})$} at 15 15
    \pinlabel {$\gamma_{i,j}$} at 110 33
    \pinlabel {$A_{i,j-1}^{-1}$} at 177 53
    \pinlabel {$A_{i,j+1}$} at 232 22
    \pinlabel {$A_{i,j+1}(A_{i,j-1}^{-1}(\gamma_{i,j}))$} at 170 17
    \endlabellist
\includegraphics[scale=1.4]{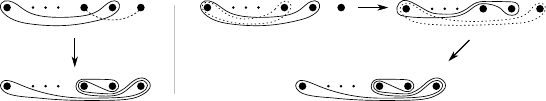}
  \end{center}
  
    \caption{$i< 2\ell -1, j = 2\ell$.  The left figure shows $\gamma_{i,j}$ and its image under $a_{\ell}^{-1}$, which is a half-twist about the dashed arc.  Starting at the top left and going clockwise, the right figure shows $\gamma_{i,j}$, $A_{i,j-1}^{-1}(\gamma_{i,j})$ and $A_{i,j+1}(A_{i,j-1}^{-1}(\gamma_{i,j}))$.  The dashed curves indicate the curves about which $A_{i,j-1}^{-1}$ and $A_{i,j+1}$ are Dehn twists.}
    \label{j2l}
\end{figure}
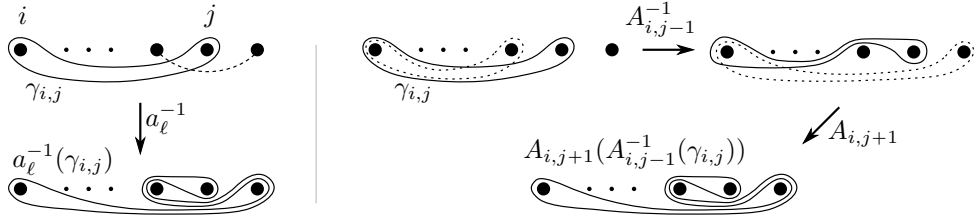

\p{The case where $i<2\ell-1$ and $j=2\ell+1$}  The curves $a_{\ell}^{-1}(\gamma_{i,j})$ and $A_{i,j}(\gamma_{i,j-2})$ are shown to be isotopic in Figure \ref{j2l+1}.  Therefore $a_{\ell}A_{i,j}a_{\ell}^{-1} = A_{i,j}^{-1}A_{i,j-2}A_{i,j}$.

 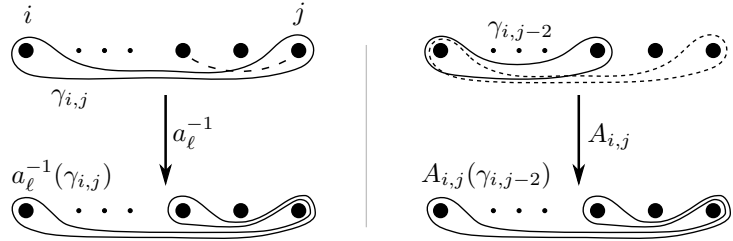
\begin{figure}[t]
  \begin{center}
\labellist\small\hair 1pt
    \pinlabel {$i$} at 4 53
    \pinlabel {$j$} at 68 53
    \pinlabel {$a_{\ell}^{-1}$} at 43 25
    \pinlabel {$\gamma_{i,j}$} at 14 33
    \pinlabel {$a_{\ell}^{-1}(\gamma_{i,j})$} at 12 15
    \pinlabel {$\gamma_{i,j-2}$} at 120 49
    \pinlabel {$A_{i,j}$} at 141 24
    \pinlabel {$A_{i,j}(\gamma_{i,j-2})$} at 112 15
    \endlabellist
\includegraphics[scale=1.6]{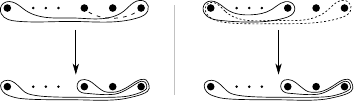}
  \end{center}
    \caption{$i< 2\ell -1, j = 2\ell + 1$.  The left figure shows $\gamma_{i,j}$ and its image under $a_{\ell}^{-1}$, which is a half-twist about the dashed arc.  The right figure shows $\gamma_{i,j-2}$ and its image under $A_{i,j}$, which is a Dehn twist about the dashed curve.}
    \label{j2l+1}
\end{figure}

 \p{The case where $i=2\ell-1$ and $j>2\ell+1$} The curves $a_{\ell}^{-1}(\gamma_{i,j})$ and $\gamma_{i+2,j}$ are shown to be isotopic in Figure \ref{i2l-1}.  Therefore $a_\ell A_{i,j} a_\ell^{-1} = A_{i+2,j}$.

\begin{figure}[t] 
\begin{center}
\labellist\small\hair 1pt
    \pinlabel {$i$} at 5 15
    \pinlabel {$j$} at 69 15
    \pinlabel {$a_{\ell}^{-1}$}  at 95 13
    \pinlabel {$\gamma_{i,j}$} at 37 13
    \pinlabel {$a_{\ell}^{-1}(\gamma_{i,j})$} at 162 13
    \endlabellist
\includegraphics[scale=1.6]{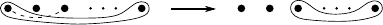}
  \end{center} 
    \caption{$i=2\ell-1, j>2\ell+1$. The curve $\gamma_{i,j}$ and its image under $a_{\ell}^{-1}$, which is a counterclockwise half-twist about the dashed arc.}
    \label{i2l-1}
\end{figure}
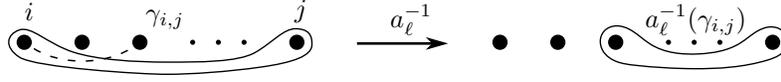

\p{The case where $i=2\ell$ and $j>2\ell+1$} The curves $a_\ell^{-1}(\gamma_{i,j})$ and $A_{i-1,i}(A_{i,i+1}^{-1}(\gamma_{i,j}))$ are shown to be isotopic in Figure \ref{i2l}.  Therefore
\[
a_{\ell}A_{i,j}a_{\ell}^{-1}=A_{i-1,i}^{-1}A_{i,i+1}A_{i,j}A_{i,i+1}^{-1}A_{i-1,i}.
\]

\begin{figure}[t]
\begin{center}
\labellist\small\hair 1pt
    \pinlabel {$i$} at 5 51
    \pinlabel {$j$} at 69 51
    \pinlabel {$a_{\ell}^{-1}$} at 45 24
    \pinlabel {$\gamma_{i,j}$} at 65 32
    \pinlabel {$a_{\ell}^{-1}(\gamma_{i,j})$} at 14 16
    \pinlabel {$\gamma_{i,j}$} at 123 31
    \pinlabel {$A_{i,i+1}^{-1}$} at 180 51
    \pinlabel {$A_{i-1,i}$} at 232 21
    \pinlabel {$A_{i-1,i}(A_{i,i+1}^{-1}(\gamma_{i,j}))$} at 175 15
    \endlabellist
\includegraphics[scale=1.4]{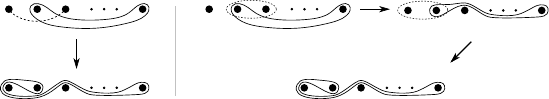}
\end{center}
\caption{$i = 2\ell, j>2\ell + 1$.  The left figure shows $\gamma_{i,j}$ and its image under $a_{\ell}^{-1}$, which is a half-twist about the dashed arc.  Starting at the top left and going clockwise, the right figure shows $\gamma_{i,j}$, $A_{i,i+1}^{-1}(\gamma_{i,j})$ and $A_{i-1,i}(A_{i,i+1}^{-1}(\gamma_{i,j}))$.  The dashed curves indicate the curves about which $A_{i,i+1}^{-1}$ and $A_{i-1,i}$ are Dehn twists.}
\label{i2l}
\end{figure}
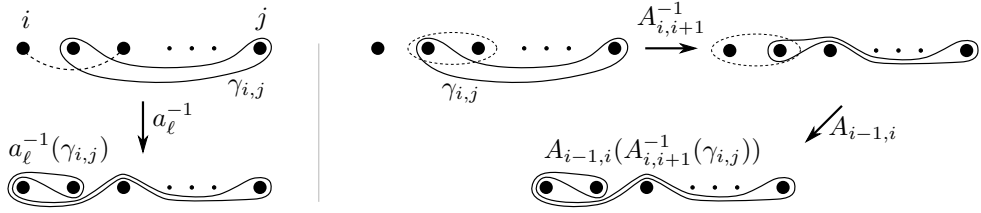


\begin{figure}[t]
\begin{center}
\labellist\small\hair 1pt
    \pinlabel {$i$} at 31 50
    \pinlabel {$j$} at 68 50
    \pinlabel {$a_{\ell}^{-1}$} at 44 22
    \pinlabel {$A_{i,j}$} at 142 22
    \pinlabel {$\gamma_{i,j}$} at 64 32
    \pinlabel {$a_{\ell}^{-1}(\gamma_{i,j})$} at 12 12
    \pinlabel {$\gamma_{i-2,j}$} at 113 32
    \pinlabel {$A_{i,j}(\gamma_{i-2,j})$} at 112 11
    \endlabellist
\includegraphics[scale=1.6]{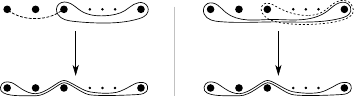}
\end{center}
\caption{$i =  2\ell +1, j > 2\ell + 1$.  The left figure shows $\gamma_{i,j}$ and its image under $a_{\ell}^{-1}$, which is a half-twist about the dashed arc.  The right figure shows $\gamma_{i-2,j}$ and its image under $A_{i,j}$, which is a Dehn twist about the dashed curve.}
\label{i2l+1proof}
\end{figure}
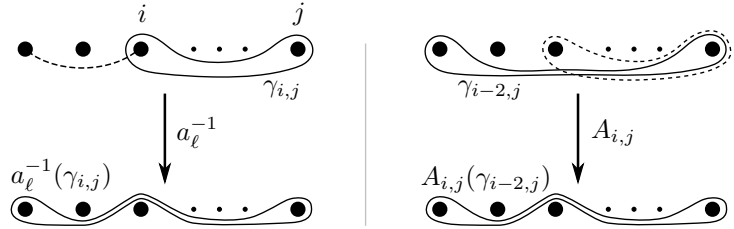

\p{The case where $i=2\ell+1$ and $j>2\ell+1$} The curves $a_\ell^{-1}(\gamma_{i,j})$ and $A_{i,j}(\gamma_{i-2,j})$ are shown to be isotopic in Figure \ref{i2l+1proof}.  Therefore $a_{\ell}A_{i,j}a_{\ell}^{-1}=A_{i,j}^{-1}A_{i-2,j}A_{i,j}$.
 \end{proof}


Next we consider conjugation of the elements $A_{i,j}$ by the generators $b_{\ell}$.  The resulting relations correspond to the remaining words in the set of conjugates of the words in $\wt S_K$ by the words in $\wt S_H$.

Let \begin{equation}\label{Zij}
Z_{i,j,\ell}=\begin{cases}
A_{i,j}&\text{ if }i < 2\ell,j > 2\ell+2\\
A_{i,j}&\text{ if } i,j > 2\ell+2\text{ or }i,j< 2\ell\\
A_{i,j+2}&\text{ if } i<2\ell,j=2\ell\\
(A_{i,j-1}^{-1}A_{i,j+1})^{-1}A_{i,j}(A_{i,j-1}^{-1}A_{i,j+1})&\text{ if } i<2\ell, j = 2\ell+1\\
A_{i,j}^{-1}A_{i,j-2}A_{i,j} &\text{ if }i<2\ell,j = 2\ell+2\\
A_{i,j+1}A_{j,j+1}A_{i,j+1}^{-1} &\text{ if }i = 2\ell, j = 2\ell+1\\ 
A_{i,j}&\text{ if }i = 2\ell, j = 2\ell+2\\
A_{i+2,j}&\text{ if }i = 2\ell,j>2\ell+2\\ 
A_{i-1,j-1}&\text{ if }i = 2\ell+1, j = 2\ell+2\\
(A_{i,i+1}^{-1}A_{i-1,i})^{-1}A_{i,j}(A_{i,i+1}^{-1}A_{i-1,i})&\text{ if }i = 2\ell+1,j>2\ell+2\\
A_{i,j}^{-1}A_{i-2,j}A_{i,j}&\text{ if }i = 2\ell+2,j>2\ell+2. \end{cases}
\end{equation}

\begin{lem}\label{conjugationb}
For $1 \leq i < j \leq 2n+1$ and $\ell \in \{1,\ldots, n\}$, let $A_{i,j}$ and $b_\ell$ be as above.  Then $$b_\ell A_{i,j}b_\ell^{-1} = Z_{i,j,\ell}$$
where the $Z_{i,j,\ell}$ are as in (\ref{Zij}).

\end{lem}
The proof of Lemma \ref{conjugationb} is the same as the proof of Lemma \ref{conjugationa} with an increase in index by 1.

\subsection{Proof of the presentation}
We are now ready to write down a presentation for $\LMod_{g,k}(\Sigma_0,\vB)$.

\begin{thm}  \label{maintheorem} 
Let $\Sigma_g$ be a surface of genus $g\geq 2$.  Let $\Sigma_g\rightarrow\Sigma_0$ be a balanced superelliptic cover of degree $k\geq 3$ with set of branch points $\vB=\vB(2n+2)$.  The subgroup $\LMod_{g,k}(\Sigma_0,\vB)$ is generated by
\begin{align*}
A_{i,j} &= (\sigma_{j-1}\sigma_{j-2}\cdots \sigma_{i+1})\sigma_i^2(\sigma_{j-1}\sigma_{j-2}\cdots \sigma_{i+1})^{-1}, 1 \leq i < j \leq 2n+1 \\
c &= \sigma_1\sigma_3\cdots \sigma_{2n-1}\sigma_{2n+1} \\
a_i &= \sigma_{2i}\sigma_{2i-1}\sigma_{2i}^{-1},\ i \in \{1,\ldots, n\}\\
b_i &= \sigma_{2i+1}\sigma_{2i}\sigma_{2i+1}^{-1},\ i \in \{1,\ldots, n\}.
\end{align*}

For $\ell \in \{1,\ldots,2n+1\}$, let $A_{\ell,2n+2}$ be defined as in Lemma \ref{Abar}.  Then \linebreak $\LMod_{g,k}(\Sigma_0,\vB)$ has defining relations:

\hspace{24pt}{\bf Commutator relations}
\begin{enumerate}
\item $\left[A_{i,j},A_{p,q}\right]=1$ where $1\leq i<j<p<q\leq 2n+1$. 
\item $\left[A_{i,q},A_{j,p}\right]=1$ where $1\leq i<j<p<q\leq 2n+1$.
\item $\left[ A_{p,q}A_{i,p}A_{p,q}^{-1}, A_{j,q}\right]=1$ where $1\leq i<j<p<q\leq 2n+1$. 
\item $[a_i,b_j] = C_{i,j}$ where $C_{i,j}$ are as in (\ref{commutators}).\\
{\bf Braid relations}
\item $A_{i,p}A_{j,p}A_{i,j}= A_{j,p} A_{i,j}A_{i,p}= A_{i,j} A_{i,p} A_{j,p}$ where $1\leq i<j<p\leq 2n+1$.
\item $a_ia_{i+1}a_i = a_{i+1}a_ia_{i+1}$ and $b_ib_{i+1}b_i = b_{i+1}b_ib_{i+1}$ for $i \in \{1,\ldots, n-1\}$.
\item $[a_i,a_j] = [b_i,b_j] = 1$ if $\abs{j-i}>1$.\\
{\bf Subsurface support}
\item $(A_{1,2}A_{1,3}\cdots A_{1,m-1})\cdots(A_{m-3,n-2}A_{m-3,n-1})(A_{m-2,m-1})=1$ for $m = 2n+2$. \\
{\bf Half twists squared are Dehn twists}
\item $a_i^2 = A_{2i-1,2i+1}$ and $b_i^2 = A_{2i,2i+2}$ for $i \in \{1,\ldots, n\}$.
\item $c^2 = A_{1,2}A_{3,4}\cdots A_{2n+1,2n+2}$.\\
{\bf Parity Flip}
\item $ca_ic^{-1}b_i^{-1} = 1$\\
{\bf Conjugation relations}
\item $cA_{i,j}c^{-1} = X_{i,j}$ where the $X_{i,j}$ are as in (\ref{Xij}).
\item $a_\ell A_{i,j}a_\ell^{-1} = Y_{i,j,\ell}$ where the $Y_{i,j,\ell}$ are as in (\ref{Yij}).
\item $b_\ell A_{i,j}b_\ell^{-1} = Z_{i,j,\ell}$ where the $Z_{i,j,\ell}$ are as in (\ref{Zij}).
\end{enumerate}
\end{thm}

\begin{proof}
We prove the elements in (\ref{gens}) are the generators of $\LMod_{p,k}(\Sigma_0,\vB)$ in Lemma \ref{lmod_gens}.

Let $\wt{R}_K$ denote the image of the relations of $\PMod(\Sigma_0,\vB(2n+2))$ in \linebreak $\LMod_{g,k}(\Sigma_0,\vB(2n+2))$.  Then $\wt R_K$ consists of the relations (1), (2), (3), (5), and (8) by Lemma \ref{pmod_pres}.  

Let $R_1$ denote the lifts in $\LMod_{p,k}(\Sigma_0,\vB)$ of the relations of $W_{2n+2}$.  The relations (4), (6), (7), (9), (10), and (11) are the relations of $R_1$ in Lemma \ref{R_1relations}.

Finally, the set $R_2$ in Lemma \ref{pres_short_exact} consists of the relations (12)-(14) as proved in lemmas \ref{conjugationc}, \ref{conjugationa}, and \ref{conjugationb}.

By Lemma \ref{pres_short_exact} the sets $\wt{R}_K, R_1,$ and $R_2$ comprise all of the relations of \linebreak $\LMod_{p,k}(\Sigma_0,\vB)$.
\end{proof}

The strategy we employed to find this presentation can be used to find a presentation for $\LMod_p(\Sigma_0,\vB)$ where $p:\Sigma_g \to \Sigma_0$ is any abelian branched cover of the sphere.  Indeed, $\LMod_p(\Sigma_0,\vB)$ can be written as a group extension of \linebreak $\widehat\Psi(\LMod_p(\Sigma_0,\vB))$ by $\PMod(\Sigma_0,\vB)$.  If one can compute a presentation for \linebreak $\widehat\Psi(\LMod_p(\Sigma_0,\vB))$, which is a subgroup of the symmetric group $S_{\abs{\vB}}$, then the generators of $\LMod_p(\Sigma_0,\vB)$ will be the lifts of the generators of $\widehat\Psi(\LMod_p(\Sigma_0,\vB))$ and the Dehn twists $A_{i,j}$ in Theorem \ref{maintheorem}.  The relations can then be found by performing the analogous computations to lemmas \ref{R_1relations}, \ref{conjugationc}, \ref{conjugationa}, and \ref{conjugationb} and applying Lemma \ref{pres_short_exact}.

\section{Abelianization}\label{abelianization_section}

In this section we will prove theorems \ref{abelianization} and \ref{betti}.  Recall that for any group $G$, $H_1(G;\ZZ) \cong G/[G,G]$.  For this section, fix $k \geq 3$ and let $p_{g,k}:\Sigma_g \to \Sigma_0$ be the balanced superelliptic cover of degree $k$.  Recall that there are $2n+2$ branch points where $n=g/(k-1)$.  The abelianization of $\LMod_{g,k}(\Sigma_0,\vB)$ depends on $n$.  For ease of notation, let $G_n = \LMod_{g,k}(\Sigma_0,\vB)$ for the remainder of this section.  Let $\phi:G_n \to G_n/[G_n,G_n]$ be the abelianization map.  Note that if $a,b \in G_n$ are in the same conjugacy class of $G_n$, then $\phi(a) = \phi(b)$.

A presentation for $G_n/[G_n,G_n]$ is given by taking a presentation for $G_n$ and adding the set of all commutators to the set of defining relations.  We begin with the presentation given in Theorem \ref{maintheorem}.

Performing Tietze transformations we may add the generators $A_{\ell,2n+2}$ for $\ell \in \{1,\ldots, 2n+1\}$ along with the relations
\[
A_{\ell,2n+2} =  ( \ol A_{1,2} \cdots \ol A_{1,2n})(\ol A_{2,3} \cdots \ol A_{2,2n}) \cdots (\ol A_{2n-2,2n-1}\ol A_{2n-2,2n})(\ol A_{2n-1,2n})
\]
where the $\ol A_{i,j}$ are as in Lemma \ref{Abar}.  

\begin{lem} \label{Aij_conj_classes}
If $j - i \equiv t - s \mod 2$, then $A_{i,j}$ is conjugate to $A_{s,t}$ in $G_n$.
\end{lem}
\begin{proof} We consider two cases: either $j-i\equiv t-s\equiv 0\mod 2$ or $j-i\equiv t-s\equiv 1\mod 2$.

{\bf Case 1:  $j - i \equiv t - s \equiv 0 \mod 2$}.\\
Let $i$ and $j$ be even.  Recall conjugation relations
\[
b_\ell A_{i,j} b_\ell^{-1} = A_{i,j}^{-1}A_{i,j-2}A_{i,j}
\]
for $i < 2\ell$ and $j = 2n+2$, and
\[
b_\ell A_{i,j} b_\ell^{-1} = A_{i,j+2}
\]
for $i < 2\ell$ and $j = 2\ell$.  Therefore for any fixed even $i$, all generators $A_{i,j}$ with even $j$ are in the same conjugacy class of $G_n$. 
 We also have the conjugation relations
\[
b_\ell A_{i,j} b_\ell^{-1} = A_{i+2,j}
\]
for $i = 2\ell$ and $j >2\ell + 2$, and
\[
b_\ell A_{i,j} b_\ell^{-1} = A_{i,j}^{-1}A_{i-2,j}A_{i,j}
\]
for $i = 2\ell+2$ and $j > 2\ell + 2$.  Therefore for any fixed even $j$, all the $A_{i,j}$ such that $i$ is even are in the same conjugacy class of $G_n$.  Then by varying $j$, we conclude that if $i,j,s,t$ are all even, then $A_{i,j}$ and $A_{s,t}$ are conjugate.

Similarly we can consider the conjugacy relations $a_\ell A_{i,j} a_\ell^{-1} = Y_{i,j,\ell}$ to conclude that if $i,j,s,t$ are all odd, then $A_{i,j}$ is conjugate to $A_{s,t}$ in $G_n$.

Observe that $cA_{1,3}c^{-1} = A_{2,4}$.  We may finally conclude that if $j - i \equiv t - s \equiv 0 \mod 2$, then $A_{i,j}$ is conjugate to $A_{s,t}$ in $G_n$.

{\bf Case 2:  $j - i \equiv t - s \equiv 1 \mod 2$.}\\
Similar to the proof of case 1 above, we use relations from the family of relations $a_\ell A_{i,j} a_\ell^{-1} = Y_{i,j,\ell}$ to conclude that for any fixed even $i$, all the $A_{i,j}$ for any odd $j$ are in the same conjugacy class of $G$.  Using relations of the form $b_\ell A_{i,j}b_\ell^{-1} = Z_{i,j,\ell}$ gives us that for any fixed odd $j$, all the $A_{i,j}$ for any even $i$ are in the same conjugacy class of $G_n$.  Therefore if $i$ and $s$ are even and $j$ and $t$ are odd, then $A_{i,j}$ and $A_{s,t}$ are conjugate in $G_n$.

Similarly, if $i$ and $s$ are odd and $j$ and $t$ are even, then $A_{i,j}$ and $A_{s,t}$ are conjugate in $G_n$.

Finally, the relation $cA_{2,3}c^{-1} = A_{2,4}^{-1}A_{1,4}A_{2,4}$ allows us to conclude that if $j - i \equiv t - s \equiv 1 \mod 2$, then $A_{i,j}$ is conjugate to $A_{s,t}$ in $G_n$, completing the proof.
\end{proof}

From now on, let $A = \phi(A_{1,2})$ and $B = \phi(A_{1,3})$.

\begin{lem} \label{AB_relation1}
For each $\ell \in \{1,\ldots,2n+1\}$, consider the relation
\[
A_{\ell,2n+2} = (\ol A_{1,2} \cdots \ol A_{1,2n})(\ol A_{2,3} \cdots \ol A_{2,2n}) \cdots (\ol A_{2n-2,2n-1} \ol A_{2n-2,2n})(\ol A_{2n-1,2n})
\]
where the $\ol A_{i,j}$ are as in Lemma \ref{Abar}.  Applying $\phi$ to each of these relations gives the relation $B^{n^2 - n} = A^{1 - n^2}$ in $G_n/[G_n,G_n]$.
\end{lem}
\begin{proof}
Fix $\ell \in \{1,\ldots,2n+1\}$ and let
\begin{align*}
\ol W &= (\ol A_{1,2} \cdots \ol A_{1,2n})(\ol A_{2,3} \cdots \ol A_{2,2n}) \cdots (\ol A_{2n-2,2n-1} \ol A_{2n-2,2n})(\ol A_{2n-1,2n}) \\
W &= (A_{1,2} \cdots A_{1,2n+1})(A_{2,3} \cdots A_{2,2n+1}) \cdots (A_{2n-1,2n}A_{2n-1,2n+1})(A_{2n,2n+1}) \\
L &= \prod_{\substack{1 \leq i < j \leq 2n+1 \\ i=\ell\text{ or }j=\ell}}  A_{i,j}.
\end{align*}
Observe that $\phi(\ol W) = \phi(W) \phi(L)^{-1}$.  By Lemma \ref{Aij_conj_classes} we have
\begin{align*}
\phi(W) &= ((AB)^n)((AB)^{n-1}A)((AB)^{n-1}) \cdots (AB)(A) \\
&= A^{2n}A^{2(n-1)} \cdots A^2 B^n B^{2(n-1)} B^{2(n-2)} \cdots B^2 \\ 
&= A^{n(n+1)}B^{n^2}
\end{align*}
since $\sum_{i=1}^{n-1} 2i = n(n-1)$.

If $\ell$ is even, $\phi(L) = A^{n+1}B^{n-1}$.  Applying $\phi$ to the relation above gives
\[
B = \phi(W) = A^{n(n+1)}B^{n^2}A^{-n-1}B^{1-n}.
\]
This rearranges to $B^{n^2 - n} = A^{1-n^2}$.

If $\ell$ is odd, $\phi(L) = A^nB^n$.  Applying $\phi$ to the relation above gives $B^{n^2 - n} = A^{1 - n^2}$.
\end{proof}

\begin{lem} \label{AB_relation2}
In the abelianization of $G_n$, $B^{n^2} = A^{-n^2 - n}$.
\end{lem}
\begin{proof}
Consider the subsurface support relation,
\[
(A_{1,2} \cdots A_{1,2n+1})(A_{2,3} \cdots A_{2,2n+1}) \cdots (A_{2n-1,2n}A_{2n-1,2n+1})(A_{2n,2n+1}) = 1.
\]
Applying $\phi$ to both sides gives $1 = A^{n(n+1)}B^{n^2}$ by the computation of $\phi(W)$ in the proof of Lemma \ref{AB_relation1}.  
\end{proof}

\begin{lem} \label{ab_conj_classes}
For all $1 \leq i,j \leq n$, $\phi(a_i) = \phi(b_j)$.
\end{lem}
\begin{proof}
By Lemma \ref{R_1relations}, we have the braid relations $(a_{i+1}^{-1}a_i)a_{i+1}(a_{i+1}^{-1}a_i)^{-1}=a_i$ for $i \in \{1,\ldots,n-1\}$ and $(b_{i+1}^{-1}b_i)b_{i+1}(b_{i+1}^{-1}b_i)^{-1}=b_i$ for all $i \in \{1,\ldots,n-1\}$.  Therefore all $\phi(a_i)=\phi(a_j)$ and $\phi(b_i)=\phi(b_j)$ for all $i,j \in \{1,\ldots,n-1\}$.  The parity flip relation $ca_1c^{-1} = b_1$ allows us to deduce that $a_i$ and $b_j$ are conjugate for all $1 \leq i,j \leq n$ and $\phi(a_i) = \phi(b_j)$.
\end{proof}

\begin{lem}\label{intermediate_pres}
The abelianization $G_n/[G_n,G_n]$ admits the presentation
\[
\langle a,d,A,B \mid B^{n^2-n} = A^{1-n^2}, B^{n^2} = A^{-n^2 - n}, a^2 = B, d^2 = A^{n+1}, \vT \rangle
\]
where $a = \phi(a_1)$, $d = \phi(c)$, $A = \phi(A_{1,2})$, $B = \phi(A_{1,3})$, and $\vT$ is the set of all commutators.
\end{lem}
\begin{proof}
Lemmas \ref{Aij_conj_classes} and \ref{ab_conj_classes} show that the elements $\phi(a_1), \phi(c), \phi(A_{1,2})$ and $\phi(A_{1,3})$ form a generating set for $G_n/[G_n,G_n]$.

Lemmas \ref{AB_relation1} and \ref{AB_relation2} show that the relations $B^{n^2 - n} = A^{1-n^2}$ and $B^{n^2} = A^{-n^2 - n}$ hold in $G_n/[G_n,G_n]$.  Applying $\phi$ to the relation $a_1^2 = A_{1,3}$ shows that $a^2 = B$.  Applying $\phi$ to the relation $c^2 = A_{1,2}A_{3,4} \cdots A_{2n+1,2n+2}$ gives the relation $d^2 = A^{n+1}$.

Lemma $\ref{AB_relation1}$ shows that for all $\ell \in \{1,\ldots,2n+1\}$, the relation
\[
A_{\ell,2n+2} = (\ol A_{1,2} \cdots \ol A_{1,2n})(\ol A_{2,3} \cdots \ol A_{2,2n}) \cdots (\ol A_{2n-2,2n-1} \ol A_{2n-2,2n})(\ol A_{2n-1,2n})
\]
is derivable from $\vT$ and $B^{n^2 - n} = A^{1 - n^2}$.  

It remains to show that in the abelianization, the relations from the presentation of $G_n$ in Theorem \ref{maintheorem} can be derived from the proposed defining relations.

The commutator relations (1)-(4) of Theorem \ref{maintheorem} all map to the identity under $\phi$.
The braid relations (5) and (7) of Theorem \ref{maintheorem} are derivable from $\vT$.  The braid relation (6) is also derivable from $\vT$ since all relations in this family take the form $a = a$ in the abelianization.  Relation (8) is derivable from $B^{n^2} = A^{-n^2 - n}$ by Lemma \ref{AB_relation2}.  Relations (9) and (10) are derivable from $a^2 = B$ and $d^2 = A^{n+1}$ respectively.  The image $\phi(ca_ic^{-1}b_i^{-1})$ is the identity by Lemma \ref{ab_conj_classes}.  Finally, the conjugation relations (12)-(14) are all of the form $A = A$ or $B = B$ in the abelianization, so they are all derivable from $\vT$.
\end{proof}

We now have everything needed to prove Theorem \ref{abelianization}.

\begin{proof}[Proof of Theorem \ref{abelianization}] 
Recall $H_1(\LMod_{g,k}(\Sigma_0,\vB);\ZZ) = G_n/[G_n,G_n]$.  Start with the presentation from Lemma \ref{intermediate_pres} making the substitution $B = a^2$ gives the presentation
\begin{equation} \label{almost_abelian_pres}
G_n/[G_n,G_n] \cong \langle a,d,A \mid a^{2n^2 - 2n} = A^{1-n^2}, a^{2n^2} = A^{-n^2 -n}, d^2 = A^{n+1}, \vT \rangle.
\end{equation}
This presentation has presentation matrix $\left[\begin{smallmatrix} 2 & -n-1 & 0 \\ 0 & n^2 - 1 & 2n^2 - 2n \\ 0 & n^2 + n & 2n^2 \end{smallmatrix}\right]$. If $n$ is odd, this matrix has Smith normal form $\left[\begin{smallmatrix}2 & 0 & 0 \\ 0 & 2 & 0 \\ 0 & 0 & 0\end{smallmatrix}\right]$ so $H_1(\LMod_{g,k}(\Sigma_0,\vB);\ZZ)\cong \ZZ/2\ZZ \times \ZZ/2\ZZ \times \ZZ$. If $n$ is even, the Smith normal form is $\left[\begin{smallmatrix}1 & 0 & 0 \\ 0 & 2 & 0 \\ 0 & 0 & 0\end{smallmatrix}\right]$ so $H_1(\LMod_{g,k}(\Sigma_0,\vB);\ZZ)\cong \ZZ/2\ZZ \times \ZZ$.
%
%
%
%
\end{proof}

The first Betti number of a group $G$ is the rank of the abelian group $H_1(G;\ZZ) = G/[G,G]$.  We have the following corollary.

Let $\hat{D}$ be the image of the deck group in $\Mod(\Sigma_g)$.  Recall that $\SMod_{g,k}(\Sigma_g)$ is the normalizer of $\hat{D}$ in $\Mod(\Sigma_g)$.

\begin{proof}[Proof of Theorem \ref{betti}]
A result of Birman and Hilden in \cite{BH} gives a short exact sequence
\[
1 \lra \ZZ/k\ZZ \lra \SMod_{g,k}(\Sigma_g) \lra \LMod_{g,k}(\Sigma_0,\vB) \lra 1.
\]
Since the abelianization functor is right exact, we have the exact sequence
\[
\ZZ/k\ZZ \lra H_1(\SMod_{g,k}(\Sigma_g);\ZZ) \lra H_1(\LMod_{g,k}(\Sigma_0,\vB);\ZZ) \lra 1.
\]
The result follows since $\ZZ/k\ZZ$ is finite and $H_1(\LMod_{g,k}(\Sigma_0,\vB);\ZZ)$ is rank 1 and non-cyclic.
\end{proof}


\bibliographystyle{plain}
\bibliography{LMCG}

\end{document}